\documentclass{article}

\usepackage[english]{babel}
\usepackage [T1]{fontenc}
\usepackage[utf8]{inputenc}
\usepackage{soul}
\usepackage{graphicx}
\usepackage{enumitem} 
\usepackage{amsmath}
\usepackage{mathrsfs}
\usepackage{braket}
\usepackage{mathtools} 
\usepackage{amsmath} 
\usepackage{amsfonts} 
\usepackage{amsthm} 
\usepackage{geometry}
\usepackage{float}
\usepackage{multicol}
\usepackage{authblk}
\usepackage[dvipsnames]{xcolor}

\usepackage{bbm}

\usepackage{appendix}

\numberwithin{equation}{section}

\newtheorem{theorem}{Theorem}[section]

\newtheorem{lemma}[theorem]{Lemma}
\newtheorem{proposition}[theorem]{Proposition}
\theoremstyle{definition}

\theoremstyle{definition}
\newtheorem{remark}[theorem]{Remark}

\title{Symmetry and quantitative stability \\
for the parallel surface fractional torsion problem}
\date{}

\author[1]{Giulio Ciraolo}

\author[2]{Serena Dipierro}

\author[2]{Giorgio Poggesi}

\author[1]{\authorcr Luigi Pollastro}

\author[2]{Enrico Valdinoci}

\affil[1]{\footnotesize Departimento di Matematica, \protect\\ Universit\`a di Milano, \protect\\ Via Cesare Saldini 50, \protect\\ Milan, I-20133, Italy,\protect\\ {\tt giulio.ciraolo@unimi.it}, \protect\\ {\tt luigi.pollastro@unimi.it}\bigskip}

\affil[2]{\footnotesize Department of Mathematics and Statistics, \protect\\
University of Western Australia, \protect\\
35 Stirling Highway, \protect\\
Crawley, Perth, WA 6009, Australia, \protect\\
{\tt serena.dipierro@uwa.edu.au}, \protect\\ {\tt giorgio.poggesi@uwa.edu.au}, \protect\\ {\tt enrico.valdinoci@uwa.edu.au}}

\newcommand{\rrr}{r}

\begin{document}
\maketitle

\begin{abstract}
We study symmetry and quantitative approximate symmetry for an overdetermined problem involving the fractional torsion problem in a bounded open set $\Omega \subset \mathbb R^n$. More precisely, we prove that if the fractional torsion function has a $C^1$ level surface which is parallel to the boundary $\partial \Omega$ then $\Omega$ is a ball.  If instead we assume that the solution is \emph{close} to a constant on a parallel surface to the boundary, then we quantitatively prove that $\Omega$ is \emph{close} to a ball. 
 Our results use techniques which are peculiar of the nonlocal case as, for instance, quantitative versions of fractional Hopf boundary point lemma and boundary Harnack estimates for antisymmetric functions. 
We also provide an application to the study of rural-urban fringes in population settlements.
\end{abstract}

\section{Introduction}

In the present paper we study an overdetermined problem involving the fractional Laplacian $(-\Delta)^s$, with~$s \in (0,1)$, which is defined for $u \in C^{\infty}_c (\mathbb{R}^n)$ as
\begin{equation*}
    (-\Delta)^s u(x) := c_{n,s} \, P.V. \int_{\mathbb{R}^n} \frac{u(x) - u(z)}{|x-z|^{n+2s}} dz,
\end{equation*} 
where
\begin{equation*}
    c_{n,s} = s \, (1-s) \, 4s \pi^{-n/2} \frac{\Gamma(n/2 + s)}{\Gamma(2-s)} 
\end{equation*}
(see for example \cite{di2012hitchhiker}).

Let $G$ be a smooth and bounded domain\footnote{In our notation,
``domain'' just means ``open set'', without any connectedness assumption.} in $\mathbb{R}^n$. We denote by $B_R$ the ball of radius $R > 0$ centered at the origin and let $\Omega$ be the ``Minkowski sum of $G$ and $B_R$'', namely
\begin{equation}\label{IIGR}
    \Omega := G + B_R := \{ x + y \ | \ x \in G, \, |y| < R \}.
\end{equation}
Our main goal is to study symmetry and quantitative stability properties for the fractional torsion problem
\begin{equation}
\label{s1eq4}
\begin{cases}
(-\Delta)^s u = 1 \quad & \textmd{in} \ \Omega,\\
u = 0 \quad &\textmd{in} \  \mathbb{R}^n \setminus \Omega,
\end{cases}
\end{equation}
with the overdetermined condition 
\begin{equation}
\label{s1eq5}
u = c \quad \textmd{on} \ \partial G.
\end{equation}
The overdetermined problem \eqref{s1eq4}-\eqref{s1eq5} was firstly studied in \cite{MagnaniniSakaguchi} for the classical Laplace operator and it was motivated by the study of invariant isothermic surfaces of a nonlinear nondegenerate fast diffusion equation. Later, in \cite{CMS} and \cite{ciraolo2016solutions} symmetry and quantitative approximate symmetry results were studied for more general operators. See also~\cite{MR2916825}
for related symmetry results
regarding the parallel surface problem.

In this manuscript we consider the nonlocal counterpart of this setting. Namely, on the one hand, by Lax-Milgram Theorem,
problem \eqref{s1eq4} admits a solution. On the other,
it is not clear whether or not a solution of~\eqref{s1eq4} exists
that also satisfies~\eqref{s1eq5}. This is a classical question in the realm of overdetermined problems and typically one can prove that a solution exists if and only if the domain satisfies some symmetry. In this context, our first main result is the following.

\begin{theorem}
\label{theorem1}
Let $G$ be an open bounded set of $\mathbb{R}^n$ with $\partial G$ of class $C^1$ and set $\Omega := G + B_R$, for some $R>0$. There exists a solution $u \in C^s(\overline{\Omega})$ of \eqref{s1eq4} satisfying the additional condition \eqref{s1eq5} if and only if $G$ (and therefore $\Omega$) is a ball.
\end{theorem}

It is clear that one implication of Theorem \ref{theorem1} is trivial. Indeed, given a ball $B = B_r(x_0)$ of radius $r > 0$ and center $x_0 \in \mathbb{R}^n$ we can compute the explicit solution $\psi_B$ of \eqref{s1eq4} with $\Omega = B$ (see for example \cite{dyda2012fractional}), which is given by
\begin{equation}\label{eq:explicit fractional torsion ball}
    \psi_B (x) = \gamma_{n,s} (r^2 - |x - x_0|^2)^s_+ ,
\end{equation}
where
\begin{equation}\label{eq:def constant gamma ns}
\gamma_{n,s} \coloneqq \frac{4^{-s}\Gamma(n/2)}{\Gamma (n/2+s) \Gamma (1+s)} . 	
\end{equation}
Since $\psi_B$ is radial, then condition \eqref{s1eq5} is automatically satisfied for any $G=B_\rho(x_0)$, with $\rho<r$. Therefore, in order to prove Theorem \ref{theorem1} it is enough to show that if $u$ is a solution to \eqref{s1eq4} satisfying \eqref{s1eq5} then $\Omega$ is a ball. In other words, we prove that if a solution of the torsion problem \eqref{s1eq4} has a level set which is parallel to $\partial \Omega$ then the domain is a ball and the solution is radially symmetric. Here we notice that the regularity assumptions required on $\partial G$ are the minimal ones in order to be able to start the moving planes procedure.

Once the symmetry result for problem \eqref{s1eq4}-\eqref{s1eq5} is achieved, one can ask for its quantitative stability counterpart (as done in \cite{ciraolo2016solutions} for the classical Laplacian case). More precisely, the question is the following: if $u$ is \emph{almost constant} on a parallel surface $\partial G$, is it true that the set $\Omega$ is \emph{almost a ball}? 
In this paper we give a positive answer to the problem by performing a quantitative analysis of the method of moving planes.

It is clear that an answer to this question depends on what we mean for \emph{almost}. In order to precisely state our result, we consider the Lipschitz seminorm $[u]_{\Gamma}$ of $u$ on a surface $\Gamma$
\begin{equation*}
    [u]_{\Gamma} := \sup_{x,y \in \Gamma, \, x \neq y} \frac{|u(x) - u(y)|}{|x-y|}
\end{equation*}
and the parameter
\begin{equation}\label{def:outradius-inradius}
    \rho (\Omega) := \inf \{ |t - s| \ | \ \exists p \in \Omega \ \mathrm{such \ that} \ B_s(p) \, \subset \Omega \subset B_t(p) \} \,,
\end{equation}
which controls how much the set $\Omega$ differs from a ball (clearly, $\rho(\Omega) = 0$ if and only if $\Omega$ is a ball). 
%
%

\begin{figure}[h]
\centering
\includegraphics[width=0.55\textwidth]{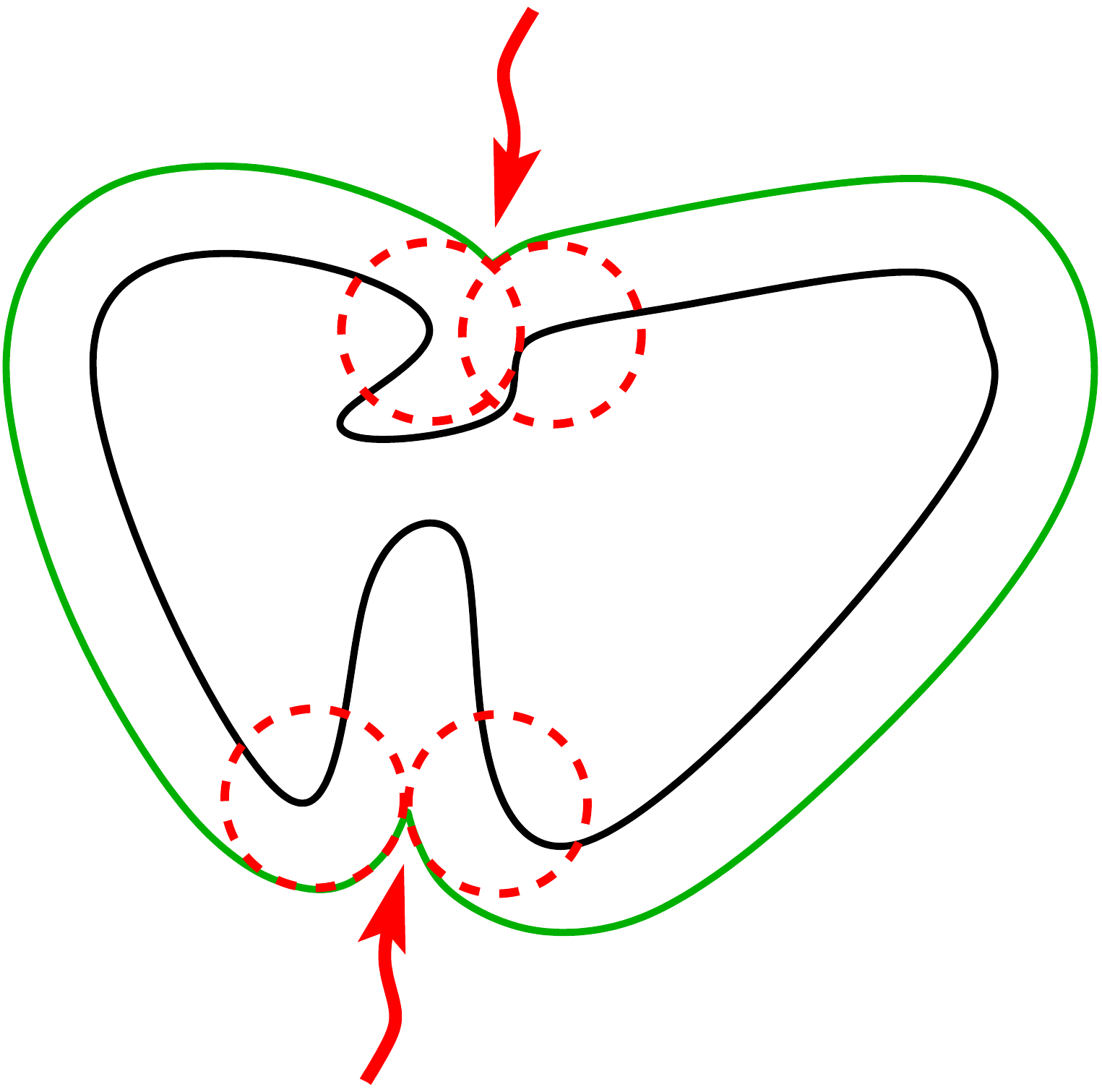}
\caption{An example in which~$G$ is~$C^\infty$ but~$\Omega$ is not~$C^1$.}
\label{LDSRIL-2}
\end{figure}

\begin{figure}[h]
\centering
\includegraphics[width=0.75\textwidth]{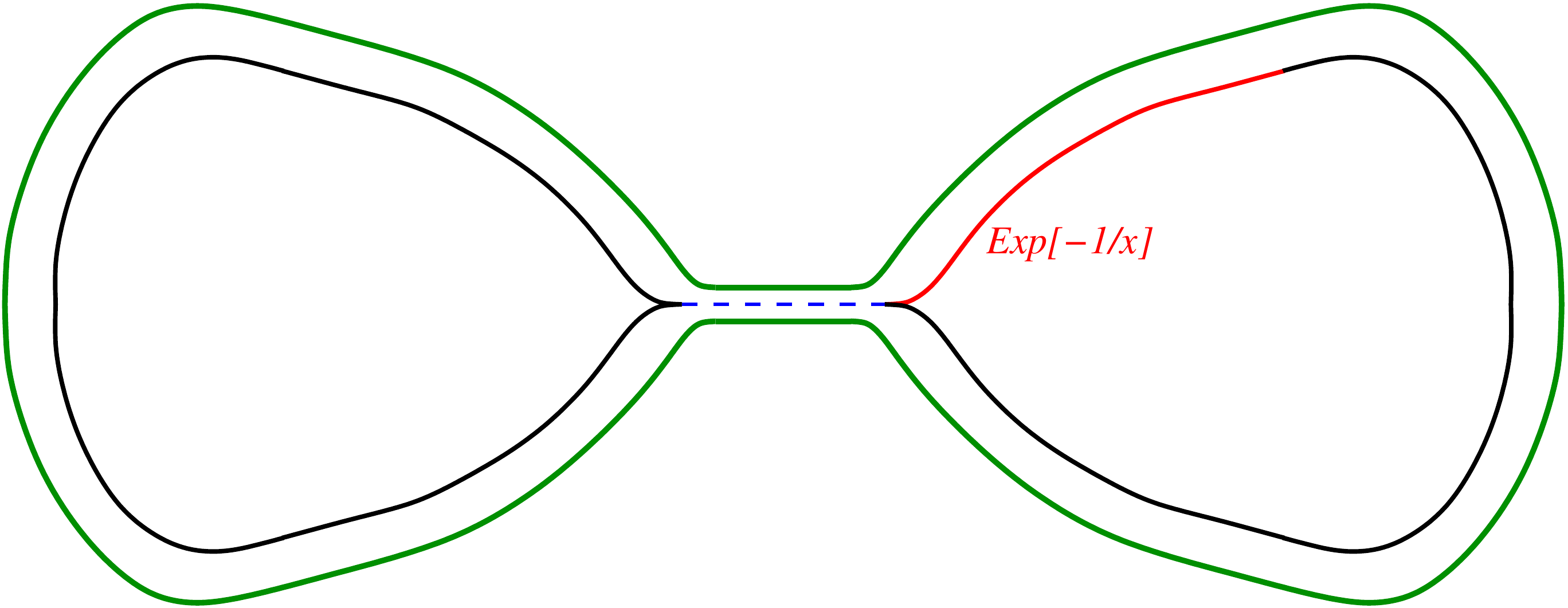}
\caption{An example in which a parallel set of $\Omega$ is not~$C^1$ even though~$\Omega$ is~$C^\infty$.}
\label{LDSRIL}
\end{figure}

Our main goal is to obtain quantitative bounds on $\rho (\Omega)$ in terms of $[u]_{\partial G}$. In particular, our second main result\footnote{We observe that there exist sets~$G$ which are $C^\infty$ but such that~$\Omega := G + B_R$ is not even~$C^1$, see e.g. Figure~\ref{LDSRIL-2}.
Moreover, we recall that well known properties of the distance function (see e.g.~\cite[Lemma~14.16]{gilbarg2015elliptic} or \cite[Theorem 5.7]{DelfourZolesio1994}) guarantee that a certain amount of regularity of $\Omega$ suffices for the regularity of its parallel sets $\left\lbrace x \in \Omega \, : \, \mathrm{dist}(x, \partial \Omega ) > R  \right\rbrace$ if $R$ is small enough, but in general~$\Omega$ can be even~$C^\infty$ and its parallel sets may fail to be~$C^1$, see e.g. Figure~\ref{LDSRIL}.

These observations justify the regularity assumptions on~$G$ and~$\Omega$ in Theorem~\ref{theorem2}.} is the following.

\begin{theorem}
\label{theorem2}
Let $G$ be an open and bounded set  of $\mathbb{R}^n$ with $\partial G$ of class $C^1$ and let $\Omega := G + B_R$. Assume that $\partial \Omega$ is of class $C^2$. Let $u \in C^2(\Omega) \cap C(\mathbb{R}^n)$ be a solution of \eqref{s1eq4}. Then, we have that
\begin{equation}
\label{s1eq8}
\rho (\Omega) \leq C_* \, [u]_{\partial G}^{\frac{1}{s +2}} ,
\end{equation}
where $C_* > 0$ is an explicit constant only depending on $n$, $s$, $R$, and the diameter $\mathrm{diam}(\Omega)$ of $\Omega$.
\end{theorem}

Hence, Theorem \ref{theorem2} asserts that
the quantity $[u]_{\partial G}$ bounds from above a pointwise measure of closeness of $\Omega$ to a ball, namely $\rho(\Omega)$. The closer $[u]_{\partial G}$ is to zero, the closer the domain $\Omega$ is to a ball (in a pointwise sense). Of course, when $[u]_{\partial G} =0$, estimate~\eqref{s1eq8} reduces to $\rho(\Omega)=0$, and therefore \eqref{def:outradius-inradius} gives that $\Omega$ is a ball: in this sense, Theorem \ref{theorem2} recovers Theorem \ref{theorem1}. 

We notice that the quantitative estimate \eqref{s1eq8} is of H\"older type and may be not optimal since we do not recover the optimal linear bound at the limit for $s \to 1$ which was obtained in \cite{ciraolo2016solutions}. The main reason
for the exponent~$\frac1{s+2}$ in~\eqref{s1eq8} is due to the technique used
to obtain our quantitative estimates, which are significantly different from the local case
and rely on detecting ``useful mass'' of the functions involved in suitable regions of the domain.

We stress that the assumption that the constant~$C_*$ in Theorem~\ref{theorem2}
depends on the diameter of~$\Omega$ is essential and cannot be removed:
an explicit example will be presented in Section~\ref{1A197}. 

 We finally notice that we do not have to make any assumption on connectedness on $G$. This is a remarkable difference with respect to the classical local case \cite{ciraolo2016solutions}. In this direction it is not difficult to see that Theorems \ref{theorem1} and \ref{theorem2} hold under weaker assumptions, in particular by assuming that the value $c$ in \eqref{s1eq5} may be different on each connected component of $G$. In Section \ref{sect8} we give further and more precise details on this result.

%
%

%
%
%

\bigskip

This paper is organized as follows. In Section~\ref{HARSEC}
we present a new
boundary Harnack result on a half ball for antisymmetric $s$-harmonic functions.
Section \ref{section_mov_planes} is devoted to the moving planes method and the proof of Theorem \ref{theorem1}; we make use of weak and strong maximum principles, as well as the boundary Harnack that we have established in Section~\ref{HARSEC}. 

In Section \ref{sect_quant_estimates} we present
a quantitative version of the fractional Hopf lemma introduced in~\cite[Proposition 3.3]{fall2015overdetermined}. Section \ref{sect_approx_1d} uses the previous results in order to get a quantitative stability estimate in one direction. Lastly, in Section \ref{sect_proofstability} we complete the proof of Theorem \ref{theorem2} by passing from the approximate symmetry in one direction to the desired quantitative symmetry result following an idea used in \cite{ciraolo2018rigidity}.

Section~\ref{1A197} presents an example that shows that the dependence
of the constant~$C_*$ in Theorem~\ref{theorem2} upon the diameter
of the domain cannot be removed. In Section \ref{sect8} we describe some possible generalization of Theorems \ref{theorem1} and \ref{theorem2}.
A technical observation of geometric type is placed in Appendix~\ref{APP:A}.


In general, we believe that the auxiliary results developed in this article, such as the boundary Harnack estimate for antisymmetric $s$-harmonic functions and the corresponding quantitative version of the fractional Hopf lemma
are of independent interest and can be used in other contexts too.

In terms of applications, in addition to the classical motivations in the study of invariant isothermic surfaces \cite{MagnaniniSakaguchi},
we mention that the overdetermined problem in~\eqref{s1eq4}
and~\eqref{s1eq5} can be inspired by questions related to population dynamics
and specifically to the determination of optimal rural-urban fringes:
in this context, our results would detect that the fair shape for an urban settlement is the circular one,
as detailed in Appendix~\ref{URBE}.

\section{Boundary Harnack inequality}\label{HARSEC}

We present here a new boundary Harnack inequality for antisymmetric $s$-harmonic functions.
{F}rom now on, we will employ the notation $H^+ := \{x_1 > 0\}$, $H^- := \{ x_1 < 0 \}$ and $T\coloneqq \{ x_1 = 0 \}$. We define $Q: \mathbb{R}^n \to \mathbb{R}^n$, with $y \mapsto y' = (-y_1,y_2,\dots,y_n)$, the reflection with respect to $T$. Moreover, for $R>0$ we call $B^+_R:= B_R \cap H^+$ and $B^-_R:= B_R \cap H^-$.

The main result towards the boundary Harnack inequality in our setting is the following:

\begin{lemma}
\label{s3lemma1}
Let $u\in C^2(B_R) \cap C(\mathbb{R}^n)$ with
\begin{equation}
\label{intconharn}
    \int_{ \mathbb{R}^n}\frac{|u(x)|}{1+|x|^{n+2s}}<+\infty
\end{equation}
be a solution of
\begin{equation*}
    \begin{cases}
    (-\Delta)^s u = 0 \quad & \text{in }  B_R,\\
    u(x') = - u(x) \quad &\text{for  every }  x \in \mathbb{R}^n,\\
    u \geq 0 \quad &  \text{in }  H^+.
    \end{cases}
\end{equation*}
There exists a constant $K > 1$ only depending on $n$ and $s$ such that, for every $z \in B_{R/2}^+$ and for every $x \in B_{R/4}(z) \cap B_R^+$ we have
\begin{equation}
\label{s3eq2}
    \frac{1}{K} \frac{u(z)}{z_1} \leq \frac{u(x)}{x_1} \leq K \frac{u(z)}{z_1}.
\end{equation}
\end{lemma}

\begin{proof}
We recall that the Poisson Kernel for the fractional Laplacian in the ball is given by (see for example \cite{bucur2015some})
\begin{equation*}
    P_{n,s} (x,y) \coloneqq c_{n,s} \bigg( \frac{R^2-|x|^2}{|y|^2-R^2} \bigg)^s \frac{1}{|x-y|^n}.
\end{equation*}
Hence for every $x \in B_R$, we have
\begin{equation*}
\begin{split}
\frac{u(x)}{c_{n,s}}\,=\,&
\int_{ \mathbb{R}^n \setminus B_R(0) }
\bigg( \frac{R^2-|x|^2}{|y|^2-R^2} \bigg)^s \frac{1}{|x-y|^n} u(y) \,dy\\=\,&
\int_{ H^+ \setminus B_R^+ }
\bigg( \frac{R^2-|x|^2}{|y|^2-R^2} \bigg)^s \frac{1}{|x-y|^n} u(y) \,dy
-
\int_{ H^+ \setminus B_R^+ }
\bigg( \frac{R^2-|x|^2}{|y|^2-R^2} \bigg)^s \frac{1}{|x-y'|^n} u(y) \,dy
\\=\,&
\int_{ H^+ \setminus B_R^+ }
\bigg( \frac{R^2-|x|^2}{|y|^2-R^2} \bigg)^s \bigg( \frac{1}{|x-y|^n} - \frac{1}{|x-y'|^n} \bigg) u(y) \,dy =: \int_{ H^+ \setminus B_R^+ } \frac{T_{n,s}(x,y)}{c_{n,s}} u(y) dy.
\end{split}
\end{equation*}
Our goal is to show that there exists a constant $K>1$ depending only on $n$ and $s$ such that 
\begin{equation}
\label{s3eq5}
   \frac{1}{K}\, \frac{x_1}{z_1} \leq \frac{T_{n,s} (x,y)}{T_{n,s} (z,y)} \leq K\, \frac{x_1}{z_1},
\end{equation}
for every $z \in B_{R/2}^+$, $x \in B_{R/4}(z) \cap B_R^+$ and $y \in H^+ \setminus B_R^+$.

We remark that once~\eqref{s3eq5} is established
the claim in~\eqref{s3eq2} readily follows, since
\begin{equation*}
    \frac{u(x)}{x_1} = \int_{ H^+ \setminus B_R^+ } \frac{T_{n,s}(x,y)}{x_1} u(y) dy \leq K \int_{ H^+ \setminus B_R^+ } \frac{T_{n,s}(z,y)}{z_1} u(y) dy = K \, \frac{u(z)}{z_1},
\end{equation*}
which is precisely the second inequality in \eqref{s3eq2}. The first inequality in \eqref{s3eq2} can be obtained similarly.\\

Now we prove \eqref{s3eq5}. We notice that
\begin{equation}
\label{s3eq7}
    \begin{split}
        \frac{T_{n,s} (x,y)}{T_{n,s} (z,y)} \,= & \, \bigg( \frac{R^2-|x|^2}{|y|^2-R^2} \bigg)^s \, \bigg( \frac{|y|^2-R^2}{R^2-|z|^2} \bigg)^s \, \bigg( \frac{1}{|x-y|^n} - \frac{1}{|x-y'|^n} \bigg) \bigg( \frac{1}{|z-y|^n} - \frac{1}{|z-y'|^n} \bigg)^{-1} 
        \\= \, &
        \bigg( \frac{R^2-|x|^2}{R^2-|z|^2} \bigg)^s \, \frac{|z-y|^n}{|x-y|^n} \, \frac{|z-y'|^n}{|x-y'|^n} \, \frac{|x-y'|^n - |x-y|^n}{|z-y'|^n - |z-y|^n} \,,
    \end{split}
\end{equation}
and we estimate the first term as follows
\begin{equation}\label{s3eq8}
    \bigg( \frac{7}{16} \bigg)^s \leq \bigg( \frac{R^2-(3R/4)^2}{R^2} \bigg)^s \leq \bigg( \frac{R^2-|x|^2}{R^2-|z|^2} \bigg)^s \leq \bigg( \frac{R^2}{R^2-(R/2)^2} \bigg)^s \leq \bigg( \frac{4}{3} \bigg)^s.
\end{equation}
Moreover, we observe that
\begin{equation}\label{s3eq9}
    \begin{split}
        & \frac{|z-y|}{|x-y|} \leq \frac{|x-y|}{|x-y|} + \frac{|x-z|}{|x-y|} \leq 1 + \frac{R/4}{R/4} = 2, \\
        & \frac{|z-y|}{|x-y|} \geq \frac{|y|-|z|}{|y|+|x|} \geq \frac{|y|-R/2}{|y|+3R/4} \geq \frac{2}{7}.
    \end{split}
\end{equation}
Now, considering the last terms in \eqref{s3eq7}, we can write
\begin{equation}
\label{s3eq10}
    \frac{|z-y'|^n}{|x-y'|^n} \, \frac{|x-y'|^n - |x-y|^n}{|z-y'|^n - |z-y|^n} =: \frac{1-\alpha^n}{1-\beta^n} \,,
\end{equation}
where
$$
\alpha = \frac{|x-y|}{|x-y'|} \quad
\text{ and }
\quad \beta=\frac{|z-y|}{|z-y'|}  \,. 
$$We observe that
\begin{equation}
\label{s3eq11}
    0 \leq \alpha^2 = \frac{|x-y|^2}{|x-y'|^2} = 1 - \frac{4x_1 y_1}{|x-y'|^2} \leq 1 \quad \text{and} \quad  0 \leq \beta^2 = \frac{|z-y|^2}{|z-y'|^2} = 1 - \frac{4z_1 y_1}{|z-y'|^2} \leq 1 \, .
\end{equation}
Going back to \eqref{s3eq10} we write
\begin{equation*}
    \frac{1-\alpha^n}{1-\beta^n} \, = \frac{(1-\alpha)(1+\alpha+\dots +\alpha^{n-1})}{(1-\beta)(1+\beta+\dots +\beta^{n-1})} \, = 
    \left[ \frac{1-\alpha^2}{1-\beta^2} \right] \frac{(1+\beta)(1+\alpha+\dots +\alpha^{n-1})}{(1+\alpha)(1+\beta+\dots +\beta^{n-1})}.
\end{equation*}
{F}rom \eqref{s3eq11} we easily get
\begin{equation}
\label{s3eq14}
    \frac{1}{2n} \leq \frac{(1+\beta)(1+\alpha+\dots +\alpha^{n-1})}{(1+\alpha)(1+\beta+\dots +\beta^{n-1})} \leq 2 n
\end{equation}
and, using estimates similar to the ones in \eqref{s3eq9},
\begin{equation}
\label{s3eq15}
    \bigg( \frac{2}{7} \bigg)^2 \, \frac{x_1}{z_1} \leq  \frac{1-\alpha^2}{1-\beta^2} \leq 4 \, \frac{x_1}{z_1}.
\end{equation}

By plugging \eqref{s3eq14} and \eqref{s3eq15} into equation \eqref{s3eq10} and then combining it with \eqref{s3eq8} and \eqref{s3eq9}, from \eqref{s3eq7} we get 
\begin{equation*}
      \bigg( \frac{7}{16} \bigg)^s \, \bigg( \frac{2}{7} \bigg)^n \, \frac{1}{2n} \, \frac{x_1}{z_1} \leq \frac{T_{n,s} (x,y)}{T_{n,s} (z,y)} \leq \bigg( \frac{4}{3} \bigg)^s \, 2^{n+3} \, n \, \frac{x_1}{z_1}
\end{equation*}
which leads to \eqref{s3eq5} if we set $K = K(n,s) := (4/3)^s (7/2)^{n+2}\, 2n > 1$. This completes the proof.
\end{proof}

As a consequence of the previous result, we get the following two propositions which provide boundary Harnack's inequalities
of independent interest:

\begin{proposition}
\label{s2prop2}
Let $u\in C^2(B_R) \cap C(\mathbb{R}^n)$ be antisymmetric w.r.t. $T$, $s$-harmonic in $B_R$, nonnegative in $H^+$ and such that \eqref{intconharn} holds. Then,
\begin{equation}
    \sup_{B_{R/2}^+} u \leq M u(\hat{x}),
\end{equation}
where $\hat{x}=\frac{R}{2} e_1$ and  $M > 0$ is a constant depending on $n$ and $s$.
\end{proposition}

\begin{proof}
Let $x_\star \in \overline{B_{R/2}^+}$ be such that
\begin{equation*}
    u(x_\star) = \sup_{B_{R/2}^+} u.
\end{equation*}
If $u(x_\star) = 0$ the result is trivial. Therefore, we can assume $u(x_\star) > 0$ and $(x_\star)_1 > 0$.

\bigskip

We now point out that any point $x \in B_{R/2}^+$ can be connected to $\hat{x}$ by a Harnack chain made at most of $3$ balls of radius $R/4$. Hence, by choosing $x_a, \, x_b \in B_{R/2}^+$ such that
\begin{equation*}
    \mathrm{dist}(x_\star, x_a) \leq R/4, \quad \mathrm{dist}(x_a, x_b) \leq R/4 \quad and \quad \mathrm{dist}(x_b, x_\star) \leq R/4
\end{equation*}
we can then apply Lemma \ref{s3lemma1} and get
\begin{equation*}
    \frac{1}{K} \, \frac{u(x_\star)}{(x_\star)_1} \leq \frac{u(x_a)}{(x_a)_1} \leq K \, \frac{u(x_b)}{(x_b)_1} \leq K^2 \, \frac{u(\hat{x})}{(\hat{x})_1}
\end{equation*}
which gives
\begin{equation*}
    \sup_{B_{R/2}^+} u = u(x_\star) \leq  K^3 \, \frac{u(\hat{x})}{R/2} \, (x_\star)_1 \leq K^3 \, u(\hat{x}),
\end{equation*}
where in the last inequality we have used that $(x_\star)_1 \leq R/2$.
\end{proof}

\begin{proposition}
Let $u, v\in C^2(B_R) \cap C(\mathbb{R}^n)$ be antisymmetric w.r.t. $T$ and satisfying \eqref{intconharn}, and assume that
\begin{equation}
    \begin{cases}
    (-\Delta)^s u = 0 = (-\Delta)^s v \quad in \ B_R^+,\\
    u,v \geq 0 \quad in \ H^+.
    \end{cases}
\end{equation}
Then
\begin{equation}
    \sup_{B_{R/2}^+} \frac{u}{v} \leq K^2 \inf_{B_{R/2}^+} \frac{u}{v},
\end{equation}
where $K = K(n,s) > 1$ is the constant given in \eqref{s3eq2}.
\end{proposition}

\begin{proof}
From Lemma \ref{s3lemma1} we have that for every $z \in B_{R/2}^+$ and every $x \in B_{R/4}(z) \cap B_R^+$
\begin{equation*}
    \frac{1}{K^2} \frac{u(z)}{v(z)} \leq \frac{u(x)}{v(x)} \leq K^2 \frac{u(z)}{v(z)}.
\end{equation*}
The proof then follows by using the Harnack chain as done in the proof of Proposition \ref{s2prop2}.
\end{proof}

\section{Moving planes method and symmetry result}  \label{section_mov_planes}

We introduce the notation needed in order to exploit the moving planes method. Given $e \in \mathbb{S}^{n-1}$, a set $E \subset \mathbb{R}^n$ and $\lambda \in \mathbb{R}$, we set
\begin{align*}
    &T_\lambda = T_\lambda^e = \{ x \in \mathbb{R}^n \, | \, x \cdot e = \lambda \} & &\textrm{a hyperplane orthogonal to } e,\\
    &H_\lambda = H_\lambda^e = \{ x \in \mathbb{R}^n \, | \, x \cdot e > \lambda \} & &\textrm{the ``positive'' half space with respect to } T_\lambda \\
    &E_\lambda = E \cap H_\lambda & &\textrm{the ``positive'' cap of } E,\\
    &x_\lambda' = x -2(x\cdot e - \lambda) \, e & &\textrm{the reflection of } x \textrm{ with respect to } T_\lambda,\\
    &Q= Q_\lambda^e : \mathbb{R}^n \to \mathbb{R}^n, x \mapsto x_\lambda' & &\textrm{the reflection with respect to } T_\lambda.
\end{align*}

If $E \subset \mathbb{R}^n$ is an open bounded set with boundary of class $C^1$ and $\Lambda_e := \sup \{ x \cdot e \, | \, x \in E \}$ it makes sense to define
\begin{equation*}
    \lambda_e = \inf \{ \lambda \in \mathbb{R} \, | \, Q(E_{\Tilde{\lambda}}) \subset E, \textrm{for all} \, \Tilde{\lambda} \in (\lambda, \Lambda_e) \}.
\end{equation*}

{F}rom this point on, given a direction $e \in \mathbb{S}^{n-1}$, we will refer to $T_{\lambda_e} = T^e$ and $E_{\lambda_e} = \widehat{E}$ as the \textit{critical hyperplane} and the \textit{critical cap} with respect to $e$, respectively, and we call $\lambda_e$ the \textit{critical value} in the direction $e$. We now recall from \cite{serrin1971symmetry} that for any given direction $e$ one of the following two conditions holds:\\

\textbf{Case 1} - The boundary of the cap reflection $Q^e(\widehat{E})$ becomes internally tangent to the boundary of $E$ at some point $P \not \in T^e$;\\

\textbf{Case 2} - the critical hyperplane $T^e$ becomes orthogonal to the boundary of $E$ at some point $Q \in T^e$.\\

Throughout this paper, the method of moving planes will be applied to the set $E=G$, where $G$ is the set appearing in \eqref{IIGR}. Hence the minimal regularity assumption that we need on $G$ is that $G$ is of class $C^1$. We also notice that, in our setting, the critical values $\lambda_e$ for $G$ are also critical values for the set $\Omega$, even if we do not need to assume further regularity on $\Omega$ in order to apply the method of moving planes. This is the reason why in Theorem \ref{theorem1} we only require that $G$ is of class $C^1$. We also notice that in Theorem \ref{theorem2} we assume that $\Omega$ is of class $C^2$, but this assumption is not needed for the application of the method of moving planes but it comes from using other tools in the proof.

\medskip

In order to prove symmetry for the problem \eqref{s1eq4} with condition \eqref{s1eq5} we will use a fractional version of the weak and strong maximum principles and a Hopf-type Lemma for antisymmetric $s$-harmonic functions. 

For $u, v \in H^s (\mathbb{R}^n)$, we consider the bilinear form induced by the fractional Laplacian
\begin{equation*}
    \mathcal{E}(u,v) := \frac{c_{n,s}}{2} \int_{\mathbb{R}^n}\int_{\mathbb{R}^n} \frac{\big( u(x) - u(y) \big) \big( v(x) - v(y) \big)}{|x-y|^{n+2s}} dxdy.
\end{equation*}

Let
\begin{equation*}
\mathcal{D}^s(\Omega) := \{ u: \mathbb{R}^n \to \mathbb{R} \quad \mathrm{measurable} \ : \ \mathcal{E}(u,\varphi) \text{ is finite in Lebesgue sense for every } \varphi \in H^s_0(\Omega) \} \, ,
\end{equation*}
where
$$
H^s_0(\Omega) := \{ u \in H^s( \mathbb{R}^n ) \ : \ u=0 \ \text{ on }  \mathbb{R}^n \setminus \Omega \}.
$$
See e.g.~\cite{MR2944369, MR3396210}
and the references therein for further information about  fractional functional spaces.

Given $g \in L^2(\Omega)$ we say that a function $u \in \mathcal{D}^s(\Omega)$ is a \emph{solution} of
\begin{equation}
\label{s3fracDir}
\begin{cases}
    (-\Delta)^s u = g \quad &in \ \Omega,\\
    u= 0 \quad &in \  \mathbb{R}^n \setminus \Omega,
\end{cases}
\end{equation}
if for all $\varphi \in H_0^s(\Omega)$ we have
\begin{equation*}
    \mathcal{E} (u, \varphi) = \int_\Omega g(x) \, \varphi(x) \, dx.
\end{equation*}

It will be useful to introduce the notion of entire antisymmetric supersolution.
Let $H \subset \mathbb{R}^n$ be a half space and let $A$ be an open set with $A \subset H$. Given $\Tilde{g} \in L^2(A)$ we say that $v \in \mathcal{D}^s (A)$ is an \emph{entire antisymmetric supersolution}\footnote{Since we are going to apply the method of moving planes, the set $A$ will typically be the intersection between the set $\Omega$ and a half space, and the function $v$ will be the difference between the solution $u$ of \eqref{s3fracDir} and its reflection with respect to an hyperplane.} of $(-\Delta)^s v = \Tilde{g}$ in $A$, if the following conditions hold:
\begin{itemize}
\item $v$ is a supersolution of $(-\Delta)^s v = \Tilde{g}$ in $A$, that is, for all $\varphi \in H_0^s(A)$, $\varphi \ge 0$ we have
\begin{equation*}
    \mathcal{E} (v, \varphi) \geq \int_A \Tilde{g}(x) \, \varphi(x) \, dx,
\end{equation*}
\item  $v \geq 0$ in $H \setminus A$ and $v$ is antisymmetric with respect to $\partial H$.
\end{itemize}
%
%

\medskip

We are now ready to prove Theorem \ref{theorem1}.

\begin{proof}[Proof of Theorem \ref{theorem1}]
We apply the method of moving planes to the set $G$. Let $e \in \mathbb{S}^{n-1}$ be a fixed direction. Without loss of generality, we can assume that $e= e_1$ and that the critical hyperplane $T$ goes through the origin (that is, $\lambda_e = 0$). We call $H^-:= \{ x_1  < 0 \}$ and consider the function
\begin{equation*}
    v(x) := u(x) - u(Q(x)) \quad \mathrm{for} \ x \in \mathbb{R}^n,
\end{equation*}
where $Q: \mathbb{R}^n \to \mathbb{R}^n$ is the reflection with respect to $T$.
We have
\begin{equation*}
\begin{cases}
(-\Delta)^s v = 0 \quad &in \ Q(\widehat{\Omega}),\\
v \geq 0 \quad &in \ H^- \setminus Q(\widehat{\Omega}),\\
v(Q(x)) = - v(x) \quad &\mathrm{for \ every} \  x \in \mathbb{R}^n.
\end{cases}
\end{equation*}
Thus, $v$ is an entire antisymmetric supersolution on $Q( \widehat{\Omega} )$. By the weak maximum principle (see~\cite[Proposition 3.1]{fall2015overdetermined}) we know that $v \geq 0$ in $H^-$. The strong maximum principle (see~\cite[Corollary 3.4]{fall2015overdetermined}) then implies that either $v > 0$ in $Q( \widehat{\Omega} )$ or $v \equiv 0$ in $\mathbb{R}^n$. We will show that the first possibility cannot occur. 

Assume by contradiction that $v > 0$ in $Q( \widehat{\Omega} )$. We need to distinguish between the two possible critical cases.\\

\textbf{Case 1} - since both $P$ and $P'$ belong to $\partial G$ and \eqref{s1eq5} holds, we immediately get that
\begin{equation*}
    v(P) = u(P) - u(P') = 0,
\end{equation*}
which is already a contradiction.\\

\textbf{Case 2} - in this case the critical hyperplane $T=\{ x_1 = 0 \}$ is orthogonal to $\partial G$ at some point $Q=(0, Q_2, \dots, Q_n)$ and therefore \eqref{s1eq5} ensures that
\begin{equation}
\label{s2eq7}
    \partial_1 v (Q) = 0.
\end{equation}
On the other hand, Lemma \ref{s3lemma1} implies the following Hopf-type inequality 
\begin{equation}\label{eq:new Hopf-type for antisymmetric for Q}
    \partial_1 v (Q) < 0,
\end{equation}
which contradicts \eqref{s2eq7} and hence \eqref{s1eq5}.
Indeed, setting $z = (-R/4, Q_2, \dots, Q_n)$ and $x=x_t = (-t, Q_2, \dots, Q_n) \in B_{R/4}(z)$, we have that
\begin{equation}\label{eq:new Hopf antisymmetric from Harnack}
	\frac{v(x_t)}{- t} \ge - \frac{ 4}{R K} v(z),
\end{equation}
where $K>1$ is a constant only depending on $n$ and $s$.
Being $z \in Q( \widehat{\Omega})$, we have that $v(z)>0$, and by letting $t$ go to 0 \eqref{eq:new Hopf-type for antisymmetric for Q} follows by \eqref{eq:new Hopf antisymmetric from Harnack}.

\medskip

This implies that $G$ (and hence $\Omega$) is symmetric with respect to the direction $e$. Since the direction $e$ is arbitrary, we easily obtain that $G$ (and hence $\Omega$) is a ball.
\end{proof}


An alternative approach to the Hopf-type inequality 
\eqref{eq:new Hopf-type for antisymmetric for Q} will be developed in a forthcoming manuscript \cite{DipPogTomVald}.


\section{A quantitative maximum principle} \label{sect_quant_estimates}

The following lemma is a quantitative version of \cite[Proposition 3.3]{fall2015overdetermined}. To state it, we adopt the notion of distance between two sets, say~$X$ and~$Y$, defined by
$$\mathrm{dist}(X,Y): = \inf \big\{|x - y |,\quad x\in X,\;y\in Y\big\}.$$
\begin{lemma}
\label{s3lemma2}
%
%
Let $B \subset H^-$ be a ball of radius $R > 0$ such that $\mathrm{dist}(B,H^+) > 0$. Let $v \in C^s(B)$  be an entire antisymmetric supersolution of
\begin{equation*}
    \begin{cases}
    (-\Delta)^s v = 0 \quad &in \ B ,\\
    v \geq 0 \quad &in \ H^-.
    \end{cases}
\end{equation*}
Let $K \subset H^-$ be a bounded set of positive measure such that $\overline{K} \subset (H^- \setminus \overline{B})$ and $\inf_K v > 0$. Then we have that
\begin{equation}
\label{s3eq19}
v \geq C \Big[  \mathrm{dist} (K,H^+) \, |K| \, \inf_K v \Big] \psi_B \quad in \ B,
\end{equation}
where $\psi_B$ is defined in~\eqref{eq:explicit fractional torsion ball},
with
\begin{equation*}
    C:= 
   \frac{2(n+2s) \,C(n,s) \,\mathrm{dist}(B,H^+)^{n+2s+1}}{
\big(\mathrm{dist}(B,H^+)^{n+2s}+C(n,s) \,\vert B \vert\, \gamma_{n,s} \,R^{2s}\big)    
    \big(
\mathrm{diam} (B) + \mathrm{diam} (K) +\mathrm{dist}(Q(K),B)\big)^{n+2s+2}}.
\end{equation*}
\end{lemma}

\begin{proof} 
%
%
%
We define
\begin{equation*}
    w(x):= \psi_B(x) - \psi_{Q(B)}(x) + \alpha \mathbbm{1}_K(x) - \alpha \mathbbm{1}_{Q(K)}(x) \quad \textrm{for }  \  x \in \mathbb{R}^n
\end{equation*}
where $\alpha > 0$ is  a parameter to be set later on, $\psi_B$ is the solution of the fractional torsion problem in $B$ and $ \mathbbm{1}_A$ is the characteristic function of a given set $A$. A direct computation shows that $w\in
\mathcal{D}^s(B)$.

The function $w$ is antisymmetric and for any nonnegative test function $\varphi \in H_0^s(B)$ we have
\begin{align*}
\mathcal{E} (w,\varphi ) &= \mathcal{E} (\psi_B,\varphi ) - \mathcal{E} (\psi_{Q(B)},\varphi ) + \alpha \, \mathcal{E} ( \mathbbm{1}_K,\varphi ) - \alpha \, \mathcal{E} ( \mathbbm{1}_{Q(K)},\varphi)\\
&= \int_B \varphi (x)dx + C(n,s) \int_B \int_{Q(B)} \frac{\psi_{Q(B)}(y) \varphi (x) }{|x-y|^{n+2s}}dydx\\
&- \alpha \  C(n,s) \int_B \int_K \frac{\varphi(x)}{|x-y|^{n+2s}}dydx + \alpha \  C(n,s) \int_B \int_{Q(K)} \frac{\varphi(x)}{|x-y|^{n+2s}}dydx\\
&\leq \int_B \varphi(x)dx \bigg[ \kappa - \alpha \ C(n,s) \int_K \bigg( \frac{1}{|x-y|^{n+2s}} - \frac{1}{|x-y'|^{n+2s}} \bigg)  \bigg],
\end{align*}
where
\begin{align*}
    \kappa = \kappa (n,s,B) =  1 + C(n,s) \,\vert B \vert  \sup_B \psi_B \sup_{x\in B, y \in H^+} \frac 1 {\vert x - y \vert^{n+2s}}<+\infty.
\end{align*}

If we set
\begin{equation}
\label{s3eq25}
    C_1 = C_1(n,s,K,B)= C(n,s) \  |K| \  \inf_{x\in B, y \in K} \bigg( \frac{1}{|x-y|^{n+2s}} - \frac{1}{|x-y'|^{n+2s}} \bigg) > 0,
\end{equation}
then
\begin{equation*}
    \mathcal{E} (w, \varphi) \leq \int_B \varphi (x) (\kappa - \alpha C_1).
\end{equation*}
By choosing $\alpha$ in such a way that $\kappa - \alpha C_1 \leq 0$, we get $(- \Delta)^s w \leq 0 $ in $B$.

For concreteness, we can thus choose
$$ \alpha:=\frac\kappa{C_1}$$
to have the previous argument in place and then set
$$
\tau \coloneqq \inf_K \frac{v}\alpha > 0
$$ 
and define 
$$
\tilde{v} (x) \coloneqq v(x) - \tau w(x)
$$ 
for every $x \in \mathbb{R}^n$. Recalling that $w$ is antisymmetric and that $w \equiv 0$ on $H^- \setminus (B \cup K)$ we have
\begin{equation*}
    \begin{cases}
(-\Delta)^s \Tilde{v} \geq 0 \quad & \text{in }   B\\
\Tilde{v} \geq 0 \quad & \text{in }   H^-\setminus B.
\end{cases}
\end{equation*}
{F}rom the weak maximum principle we then get that $\Tilde{v} \geq 0$ in $B$ and, in particular,
\begin{equation}
\label{s3eq28}
v \geq \tau \psi_B \quad \text{in }  B.
\end{equation}

For every $x \in B$ and every $y \in K$ we compute
\begin{equation*}
    \begin{split}
        \frac{1}{|x-y|^{n+2s}} - \frac{1}{|x-y'|^{n+2s}} 
        &= \frac{n+2s}{2}  \int_{|x-y|^2}^{|x-y'|^2}  t^{-\frac{n+2s+2}{2}} dt
        \\
&\geq \frac{n+2s}{2} \big( |x-y'|^2 - |x-y|^2 \big) |x-y'|^{-(n+2s+2)}
\\
& \geq\frac{n+2s}{2} 4x_1 y_1 |x-y'|^{-(n+2s+2)} .
    \end{split}
\end{equation*}
Moreover, for all~$x \in B$ and~$y \in K$, 
 $$ \vert x - y'\vert \leq  \mathrm{diam} (B) + \mathrm{diam} (K) +\mathrm{dist}(Q(K),B)$$
 and consequently
\begin{equation*}
        \frac{1}{|x-y|^{n+2s}} - \frac{1}{|x-y'|^{n+2s}} 
\geq
\frac{2(n+2s) \mathrm{dist}(B,H^+)\,\mathrm{dist}(K,H^+)}{\big(
\mathrm{diam} (B) + \mathrm{diam} (K) +\mathrm{dist}(Q(K),B)\big)^{n+2s+2}}.
\end{equation*} 
Hence, by~\eqref{s3eq25},
$$ C_1\ge\frac{2(n+2s) \,C(n,s) \,  |K| \,\mathrm{dist}(B,H^+)\,\mathrm{dist}(K,H^+)}{\big(
\mathrm{diam} (B) + \mathrm{diam} (K) +\mathrm{dist}(Q(K),B)\big)^{n+2s+2}.
}$$
As a result,
\begin{equation*}
    \tau =\frac{C_1}\kappa \inf_K v \geq \frac{2(n+2s) \,C(n,s) \,  |K| \,\mathrm{dist}(B,H^+)\,\mathrm{dist}(K,H^+)}{\kappa\,\big(
\mathrm{diam} (B) + \mathrm{diam} (K) +\mathrm{dist}(Q(K),B)\big)^{n+2s+2}} \inf_K v.
\end{equation*}
We also observe, owing to~\eqref{eq:explicit fractional torsion ball}, that
$$ \sup_B \psi_B (x) = \gamma_{n,s} \,R^{2s}
$$
and therefore
\begin{eqnarray*}\kappa&=&1 + C(n,s) \,\vert B \vert\, \gamma_{n,s} \,R^{2s}\, \sup_{x\in B, y \in H^+} \frac 1 {\vert x - y \vert^{n+2s}}\\&\le&
1 + \frac{C(n,s) \,\vert B \vert\, \gamma_{n,s} \,R^{2s}}{\mathrm{dist}(B,H^+)^{n+2s}}\\&=&
\frac{\mathrm{dist}(B,H^+)^{n+2s}+C(n,s) \,\vert B \vert\, \gamma_{n,s} \,R^{2s}}{\mathrm{dist}(B,H^+)^{n+2s}}
.\end{eqnarray*}
Accordingly,
\begin{equation*}
    \tau \geq \frac{2(n+2s) \,C(n,s) \,  |K| \,\mathrm{dist}(B,H^+)^{n+2s+1}\,\mathrm{dist}(K,H^+)}{
\big(\mathrm{dist}(B,H^+)^{n+2s}+C(n,s) \,\vert B \vert\, \gamma_{n,s} \,R^{2s}\big)    
    \big(
\mathrm{diam} (B) + \mathrm{diam} (K) +\mathrm{dist}(Q(K),B)\big)^{n+2s+2}} \inf_K v.
\end{equation*}
Thus, the desired conclusion follows from \eqref{s3eq28}.
\end{proof}

\section{Almost symmetry in one direction} \label{sect_approx_1d}

As customary, we say that a bounded domain $\Omega \subset \mathbb{R}^n$ satisfies the \emph{uniform interior ball condition} if there exists a radius $\rrr_\Omega > 0$ such that for every point $x_0 \in \partial \Omega$ we can find a ball $B_i \subset \Omega$ of radius $\rrr_\Omega$ with $\overline{B_i} \cap \Omega^c = \{ x_0 \}$.

In the next subsection, we collect some useful technical lemmas which hold true for domains satisfying such a condition.

\subsection{Preliminaries: some results for domains satisfying the uniform interior ball condition}

As noticed in \cite{CPY, MP2021}, the following simple explicit bound for the perimeter holds true.
\begin{lemma}[A general simple upper bound for the perimeter, \cite{CPY, MP2021}]
	\label{lem:upperboundperimeter}
	Let $D \subset \mathbb{R}^n$ be a bounded domain with boundary of class $C^{1,\alpha}$, with $0 < \alpha \le 1$. If $D$ satisfies the \emph{uniform interior ball condition} with radius $\rrr_D$, the we have that
	\begin{equation}\label{eq:upper bound perimeter}
		| \partial D | \le \frac{n | D | }{ \rrr_D }.
	\end{equation}
\end{lemma}
\begin{proof}
	By following \cite{MP2021}, the desired bound can be easily obtained by considering the solution $f \in C^{1, \alpha}(\overline{D})$ to
	$$
	\Delta f = n \, \text{ in } D, \quad f = 0 \, \text{ on } \partial D ,
	$$
	and putting together the identity
	$$
	n |D| = \int_{ \partial D } \partial_\nu f \, d \mathcal{H}^{n-1} , \text{ where } \partial_\nu \text{ denotes the outer normal derivative}, 
	$$
	with the Hopf-type inequality 
	$$ \partial_\nu f \ge \rrr_D , $$
which can be found in \cite[Theorem 3.10]{MP}.

We mention that a more general version of the bound \eqref{eq:upper bound perimeter} remains true even without assuming the uniform interior ball condition, at the cost of replacing the radius $\rrr_D$ of the ball condition with a parameter associated to the (weaker) pseudoball condition, which is always verified by $C^{1,\alpha}$ domains: see \cite[Remark 1.1]{CPY} and the last displayed inequality in the proof of \cite[Corollary 2.1]{CPY}.
\end{proof}

The previous result is useful to prove the following.

\begin{lemma}
\label{s4lemma3}
Let $\Omega \subset \mathbb{R}^n$ be a bounded domain with $\partial \Omega$ of class $C^2$. For $\delta > 0$, we set 
\begin{equation}\label{def:Adelta}
A_{\delta} := \{ x \in \Omega \ | \ \mathrm{dist}(x,\partial \Omega) < \delta \} .
\end{equation}
Then, we have that
\begin{equation}
\label{s4eq31}
    |A_\delta| \leq c \, \delta , \quad \text{ with } \quad  c:= \frac{ 2 n  | \Omega | }{ \rrr_\Omega },
\end{equation}
where $\rrr_\Omega$ is the radius of the uniform interior ball condition of $\Omega$.
\end{lemma}

We recall that if a domain has boundary of class $C^2$, then it satisfies a uniform interior ball condition.

\begin{proof}[Proof of Lemma \ref{s4lemma3}]
We set $d_{\partial \Omega}(x) := \mathrm{dist}(x,\partial \Omega)$ for $x \in \Omega$. For $\delta \ge 0$, we define 
$$
V_\delta := \{ x \in \Omega \ | \ d_{\partial \Omega}(x) > \delta \}
\quad \text{ and } \quad
\Gamma_\delta := \{ x \in \Omega \ | \ d_{\partial \Omega}(x) = \delta \} \,.
$$
It is well-known that $d_{\partial \Omega} \in C^2(A_{ \rrr_\Omega })$ (see, e.g., \cite[Lemma 14.16]{gilbarg2015elliptic}).
%
%

We first prove the claim in the case $0 \le \delta \le \rrr_\Omega / 2$.
{F}rom the coarea formula we obtain
\begin{equation}
\label{s4eq32}
    |A_{\delta}| = \int_{A_\delta} 1 \, dx = \int_{A_\delta} |\nabla d_{\partial \Omega}(x)| \, dx = \int_{0}^{\delta} \bigg( \int_{A_{\delta} \cap d_{\partial \Omega}^{-1}(t)} d\mathcal{H}^{n-1} \bigg) \, dt = \int_0^\delta |\Gamma_t| \, dt.
\end{equation}
Since $t \le \delta \le \rrr_\Omega / 2$, we have that $V_t$ is a bounded domain satisfying the uniform interior touching ball condition with radius $\rrr_\Omega /2$, and with boundary $\Gamma_t$ of class $C^2$. Thus, we can apply Lemma \ref{lem:upperboundperimeter} with $D:=V_t$ to get that
\begin{equation}\label{eq:upperbound gammat}
|\Gamma_t| \le \frac{2  n |V_t|}{\rrr_\Omega} \le \frac{2 n | \Omega |}{\rrr_\Omega} ,
\end{equation}
where the last inequality follows by the inclusion $V_t \subseteq \Omega$.
Combining \eqref{s4eq32} with \eqref{eq:upperbound gammat} immediately gives \eqref{s4eq31}, for any $0 \le \delta \le \rrr_\Omega / 2$.

On the other hand, if $\delta \ge \rrr_\Omega / 2$, we easily find that
$$
|A_\delta| \le | \Omega | \le \left[ \frac{2 |\Omega | }{ \rrr_\Omega } \right] \delta ,
$$
where the first inequality follows by the inclusion 
$$
A_\delta \subseteq \Omega , \text{ for any } \delta \ge 0 .
$$
Thus, \eqref{s4eq31} still holds true.
\end{proof}

We now detect an optimal growth of the solution to \eqref{s1eq4} from the boundary, by generalizing \cite[Lemma 3.1]{MP2} to the fractional setting.
	\begin{lemma}
		\label{lem:relationdist_fractional optimal growth}
		Let $u$ satisfy~\eqref{s1eq4} and let $\gamma_{n,s}$ be the constant defined in \eqref{eq:def constant gamma ns}.
		%
		%
		Then,
		\begin{equation}\label{instgr}
			u(x) \ge 
			\gamma_{n,s} \,\mathrm{dist}(x, \partial \Omega)^{2s} 
			\quad \mbox{ for every } \ x \in \Omega .
		\end{equation}
		
		Moreover, if $\Omega$ is of class~$C^1$
		and satisfies the uniform interior sphere condition with radius $\rrr_\Omega$, then it holds that
		\begin{equation}
			\label{eq:relationdist}
			u(x) \ge \gamma_{n,s} \, \rrr_\Omega^s \,\mathrm{dist}(x, \partial \Omega)^s  \quad \mbox{ for every } \ x \in  \Omega .
		\end{equation}
	\end{lemma}
	
	\begin{proof}
		Let~$x \in \Omega$ and set~$r:=\mathrm{dist} (x, \partial \Omega)$.
		We consider
		$$
		\psi (y) := \gamma_{n,s} \left(  r^2 - | y - x |^2 \right)_+^s ,
		$$
		which satisfies the fractional torsion problem in $B_r(x)$, namely
		\begin{equation}\label{eq:torsionball}
			\begin{cases}
				(- \Delta)^s \psi = 1 \quad & \text{ in } B_r (x), 
				\\
				\psi =0 \quad & \text{ on } \mathbb{R}^n \setminus B_r (x) .
			\end{cases}
		\end{equation}
		By the comparison principle (see~\cite[Remark 3.2]{fall2015overdetermined}), we have that $ u \ge \psi$ on $\overline{B_r (x)}$.
		In particular, at the center $x$ of $B_r (x)$, we have that  
		$$
		u(x) \ge \psi (x) = \gamma_{ n , s }\,\mathrm{dist}(x, \partial \Omega )^{2s} ,
		$$ 
		and \eqref{instgr} follows.
		
		Notice that~\eqref{eq:relationdist} follows
		from~\eqref{instgr} if~$\mathrm{dist} (x, \partial \Omega ) \ge \rrr_\Omega $.
		Hence, from now on, we can suppose that
		\begin{equation}\label{eq62347on}
			\mathrm{dist} (x, \partial \Omega ) < \rrr_\Omega .
		\end{equation}
		Let $\bar x$ be the closest point in $\partial \Omega $ to $x$ and call $\tilde B \subset \Omega$
		the ball of radius $\rrr_\Omega$ touching~$\partial \Omega $ at~$\bar x$ and containing $x$.
		Up to a translation, we can always suppose that 
		\begin{equation}\label{8686s2}
			{\mbox{the center of the ball $\tilde B$ is the origin.}}
		\end{equation}
		Now, we let~$\tilde{\psi}$ 
		be the solution of \eqref{eq:torsionball} in $\tilde B$, that is
		$\tilde{\psi} (y):= \gamma_{n,s} \left(  \rrr_\Omega^2 - |y|^2 \right)_+^s$.
		By comparison (\cite[Remark 3.2]{fall2015overdetermined}), we have that $u \ge \tilde{\psi}$ in~$\tilde B$, and hence, being $x\in \tilde B$,
		\begin{equation}\label{817928y3r}
			u(x) \ge \gamma_{n,s} \, ( \rrr_\Omega^2 - |x|^2 )_+^s =
			\gamma_{n,s} \,( \rrr_\Omega + |x| )^s ( \rrr_\Omega -|x|)_+^s \ge \gamma_{n,s} \, \rrr_\Omega^s  \, ( \rrr_\Omega - |x|)^s .
		\end{equation}
		Moreover, from~\eqref{8686s2},
		$$ \rrr_\Omega -|x|=\mathrm{dist} (x, \partial \Omega).$$
		This and~\eqref{817928y3r}
		give~\eqref{eq:relationdist}, as desired.
	\end{proof}

\subsection{Almost symmetry in one direction}\label{subsec:Almost symmetry in one direction}
{F}rom now on, we let $\Omega \coloneqq G + B_R(0)$, with $G \subseteq \mathbb{R}^n$ bounded, with $\partial G$ of class $C^1$ and $\partial \Omega$ of class $C^2$.
%
%
\begin{remark}[On the constants in the quantitative estimates]
	{\rm
The constants in all of our quantitative estimates can be explicitly computed and only depend on $n$, $s$, $R$, and $\mathrm{diam}(\Omega)$. In some of the intermediate results, the parameter $| \Omega |$ may appear. It is clear that such a parameter can be removed thanks to the bounds
\begin{equation}\label{eq:trivial bounds volume}
	\frac{\omega_n}{n} R^n \le | \Omega | \le \frac{\omega_n}{n} \mathrm{diam}(\Omega)^n , \text{ where } \frac{\omega_n}{n} \text{ is the volume of the unit ball in } \mathbb{R}^n,
\end{equation}
which easily hold true in light of the monotonicity of the volume with respect to inclusion.

We remark that the estimates of the previous subsection also depend on the radius $\rrr_\Omega$ of the uniform interior ball condition associated to $\Omega$. Nevertheless, from now on, we have that
\begin{equation}\label{eq:oss relationraggiointerno e R}
\rrr_\Omega := R ,
\end{equation}
by the definition of $\Omega \coloneqq G + B_R(0)$
}
\end{remark}

\medskip

We apply the method of moving planes to the set $G$. Hence, we fix a direction $e=e_1$ and assume the associated critical hyperplane to be $T = \{x_1 = 0\}$, with $Q: \mathbb{R}^n \to \mathbb{R}^n, x \mapsto x'$ the reflection with respect to $T$. For the proofs of the next two lemmas we will use the following notation: we set for $t \geq 0$
\begin{equation}\label{NOTAZ}
    G_t \coloneqq G + B_t(0), \quad \widehat{G_t} \coloneqq G_t \cap H^+, \quad 
    G_t^- \coloneqq G_t\cap H^- \quad U_t \coloneqq Q(\widehat{G_t}).
\end{equation}
Note that $\Omega = G_R$. \\

Let $u \in C^2(\Omega) \cap C(\mathbb{R}^n)$ be a solution of \eqref{s1eq4}. For every $x \in \mathbb{R}^n$, we set 
$$
v(x):= u(x) - u(x') \,.
$$

%
%

\begin{lemma}
\label{s4lemma1}
Given $P \in U_R$ with $B=B_{R/8} (P)$ such that $\mathrm{dist}(B, \partial U_R) \ge R/8$,
we have that
\begin{equation}
\label{NEWs4eq3}
|\Omega^- \setminus U_R| \leq \Tilde{C} \, v(P)^{\frac{1}{2+s}} ,
\end{equation}
where $\Tilde{C} > 0$ is an explicit constant depending only on $n$, $s$, $R$, and $\mathrm{diam}(\Omega)$.
\end{lemma}

\begin{proof}
For $\delta \ge 0$, we set $K_\delta := (\Omega^- \setminus U_R) \setminus (E_\delta \cup F_\delta)$,
where
$$
E_\delta := A_\delta \cap (\Omega^- \setminus U_R) \quad \text{ with } A_\delta \text{ as defined in \eqref{def:Adelta}} ,  
$$
$$
F_\delta:= \{ x \in \Omega^- \setminus U_R \, : \, \mathrm{dist}(x, T) < \delta  \} .
$$

With our choice of $B$ clearly $\mathrm{dist} (B,H^+) \geq \mathrm{dist} (B,\partial U_R) \ge R/8$ and therefore, an application of Lemma \ref{s3lemma2} with $B := B_{R/8} (P)$ and $K := K_\delta$ gives that 
%
%
\begin{equation}\label{s4eq4}
	v \geq \overset{\star}{C} \,  \big[  \mathrm{dist} (K_\delta ,H^+) \, | K_\delta | \, \inf_{K_\delta} v \big] \psi_B \quad in \ B,
\end{equation}
holds true for a suitable explicit~$\overset{\star}{C}  > 0$, depending only on $n$, $s$, $R$, and $\mathrm{diam}(\Omega)$. Here, we used that in the present situation $K \subset \Omega$ and $B \subset U_R$.
 
Now looking at $K_\delta$ we have $K_\delta \subseteq (\Omega^- \setminus U_R) \subseteq H^-$ and so $\mathrm{dist} (K_\delta,B) \ge R/8$. Moreover, since $K_\delta \subseteq (G_{R-\delta}^- \setminus U_R)$ we have that $v(x) = u(x) > 0$ for every $x \in K_\delta$;
hence, \eqref{eq:relationdist} and \eqref{eq:oss relationraggiointerno e R} give that 
\begin{equation}\label{eq:crescita di v dal bordo}
\inf_K v \geq \left[ \gamma_{n,s} R^s \right] \, \delta^s .
\end{equation}
Also, since $K_\delta \subseteq (\Omega^- \setminus U_{R}) \setminus F_\delta$, then 
\begin{equation}\label{eq:dist Kdelta da H+}
\mathrm{dist} (K_\delta,H^+) \geq \delta .
\end{equation}

Clearly, 
\begin{equation*}
 |K_\delta| = |\Omega^- \setminus U_R| - |E_\delta \cup F_\delta| \ge |\Omega^- \setminus U_R| - (|E_\delta| + | F_\delta |) .
 \end{equation*}
Since $E_\delta \subseteq A_\delta$, Lemma \ref{s4lemma3} gives that 
$$ |E_\delta| \le \left[\frac{2 n |\Omega|}{R} \right]  \delta , $$ 
where we also used \eqref{eq:oss relationraggiointerno e R}. Also, by definition of $F_\delta$, it is trivial to check that
\begin{equation*}
	|F_\delta| \le  \mathrm{diam}( \Omega )^{n-1} \delta .
\end{equation*}
Putting together the last three displayed formulas we conclude that
\begin{equation}\label{eq:stima intermedia volume Kdelta}
	|K_\delta| \ge |\Omega^- \setminus U_R| - \tilde{c} \,  \delta , 
	\quad \text{ with } \quad
	\tilde{c} := \frac{2 \omega_n \mathrm{diam}(\Omega)^n }{ R } +  \mathrm{diam}( \Omega )^{n-1} .
\end{equation}
Here, we also used the second inequality in \eqref{eq:trivial bounds volume} to remove the dependence on $| \Omega |$ in the constant $\tilde{c}$.

Putting together \eqref{s4eq4}, \eqref{eq:crescita di v dal bordo}, \eqref{eq:dist Kdelta da H+}, \eqref{eq:stima intermedia volume Kdelta}, and that $\psi_B (P) = \gamma_{n,s} (R/8)^{2s} $ (by \eqref{eq:explicit fractional torsion ball} with $x_0 := P$), we get that
\begin{equation*}
    v(P) \geq \overset{\star \star}{C} \, \delta^{1+s} \big( |\Omega^- \setminus U_R| - \tilde{c} \, \delta \big)  \quad \text{with} 
\quad \overset{\star \star}{C} := \overset{\star}{C} \, \left[ \gamma_{n,s} R^s \right] \, (R/8)^{2s} \, \gamma_{n,s} ,
\end{equation*}
that is:
\begin{equation*}
|\Omega^- \setminus U_R| \le \frac{v(p)}{ \overset{\star \star}{C} } \delta^{-(1+s)} + \tilde{c} \, \delta.
\end{equation*}
By minimizing in $\delta$ the right-hand side of the last inequality, we can conveniently choose
\begin{equation}
\delta := \left[ \frac{ (1+s)  v(p)}{\overset{\star \star}{C} \, \tilde{c} } \right]^{ \frac{1}{2+s} }
\end{equation}
and obtain that \eqref{NEWs4eq3} holds true with
$$\Tilde{C}:= \left[ \frac{ (1+s) \tilde{c}^{1+s} }{\overset{\star \star}{C}  } \right]^{ \frac{1}{2+s} } .$$
\end{proof}

The next lemma uses the previous result to get a stability estimate in one specific direction.

\begin{lemma}[Almost symmetry in one direction]
\label{s4lemma2}
We have that 
\begin{equation}
\label{s4eq6}
    |\Omega \setminus Q(\Omega)| \ \leq \overline{C} \, [ u ]_{\partial G}^{\frac{1}{ 2+s }} ,
\end{equation}
where $\overline{C} > 0$ is an explicit constant only depending on $n$, $s$, $R$, and $\mathrm{diam}(\Omega)$.
\end{lemma}

\begin{proof}
We apply the method of moving planes to $G$ in the direction $e_1$. We need to distinguish between some cases.
\\
\textbf{Case 1 -} $U_0$ is internally tangent to $G$ at a point $P$ which is not on $T$. We distinguish two subcases, according to the distance of $P$ from $T$.

\textbf{Case 1a -} We assume $\mathrm{dist}(P,T) > R/8 $. Since $P \in \partial G \cap \partial U_0$ we have
\begin{equation*}
    v(P) = u(P) - u(P') \leq  [u]_{\partial G} \  \mathrm{diam} (\Omega).
\end{equation*}
We then apply Lemma \ref{s4lemma1} 
%
%
to obtain that
\begin{equation*}
|\Omega^- \setminus U_R| \leq  \Tilde{C} \,\mathrm{diam}(\Omega)^{\frac{1}{2+s}} [u]_{\partial G}^{\frac{1}{2+s}}. 
\end{equation*}
%

\textbf{Case 1b -} $P \in \partial G \cap \partial U_0$ such that $\mathrm{dist}(P,T) \leq R/8$. 

{F}rom the definitions of $v$ and $[u]_{\partial G}$, we have that
\begin{equation}
	\label{eq:nuova per caso 1b}
	  \frac{v(P)}{(-P_1)} = \frac{2 ( u(P) - u(P') ) }{\mathrm{dist} (P, P') } \leq 2  [u]_{\partial G} ,
\end{equation}
where we adopted the notation $P=(P_1, P_2, \dots, P_n)$.

As noticed in item (ii) of Lemma \ref{lem: a technical simple lemma}, we have that $B_R(P) \subset U_R \cup \left[ \Omega \cap ( H^+ \cup T )\right]$.

We set $\widehat{P} \coloneqq (0, P_2, \dots, P_n)$ the projection of $P$ on the hyperplane $T$. We then set $\overline{P} := ( - R/4 , P_2, \dots, P_n)$ so that $- \overline{P}_1 = \mathrm{dist}(\overline{P},T) = R/4$. Using Lemma \ref{s3lemma1} with $B_R:= B_{R/2}( \widehat{P} )$, we see that
\begin{equation}
\label{s4eq10}
    \frac{ 4 }{ R } \,  v(\overline{P}) = \frac{v(\overline{P})}{( - \overline{P})_1} \leq K \  \frac{v(P)}{ ( - P_1 ) }.
\end{equation}
Putting together \eqref{eq:nuova per caso 1b} and \eqref{s4eq10} gives that
$$
v(\overline{P}) \le \frac{R}{2} K [u]_{\partial G} ,
$$
and hence an application of Lemma \ref{s4lemma1} with $ P:=\overline{P}$ leads to
\begin{equation}\label{eq:New overline C}
|\Omega^- \setminus U_R| \leq  \Tilde{C} \left( \frac{R}{2} K \right)^{\frac{1}{2+s} }  [u]_{\partial G}^{\frac{1}{2+s} }.
\end{equation}
\\
\textbf{Case 2 -} $T$ is orthogonal to the boundary of $G$ at some point $Q$.

Again, in light of item (ii) of Lemma \ref{lem: a technical simple lemma}, we have that $B_R(Q) \subset U_R \cup \left[ \Omega \cap ( H^+ \cup T )\right]$.

We choose $\overline{P} := (-R/4 , Q_2, \dots, Q_n ) $ so that $- \overline{P}_1 = \mathrm{dist}(\overline{P},T) = R/4$. 

Using Lemma \ref{s3lemma1} with $B_R:= B_R(Q)$, for every $y = (y_1, Q_2, \dots, Q_n) \in B_{R/4}(\overline{P})$ we obtain that
%
%
\begin{equation*}
    \frac{v(\overline{P})}{( - \overline{P})_1} \leq K \,  \frac{v(y)}{ ( - y_1 ) } \leq K \,  [u]_{\partial G} ,
\end{equation*}
and hence
$$
v(\overline{P}) \le \frac{R}{4} K \,  [u]_{\partial G} .
$$
Again, we apply Lemma \ref{s4lemma1} with $P:= \overline{P}$, and we get that
\begin{equation*}
    |\Omega^- \setminus U_R| \leq  \Tilde{C} \, \left( \frac{R}{4} \, K \right)^{\frac{1}{2+s}} \,
    [u]_{\partial G}^{\frac{1}{2+s}}.
\end{equation*}

In all cases, \eqref{s4eq6} holds true with
$$\overline{C}:= \Tilde{C} \, \left( \max \left\lbrace \mathrm{diam}(\Omega),  \frac{R}{2} \, K \right\rbrace \right)^{\frac{1}{2+s}} .$$
This completes the proof.
%
%
\end{proof}

\section{Stability result} \label{sect_proofstability}

For the proof of the following lemma we closely follow \cite[Lemma 4.1]{ciraolo2018rigidity}. The idea is the following: for a given direction $e \in \mathcal{S}^{n-1}$ we slice the set $\Omega$ in a (finite number of) sections depending on the critical value $\lambda_e$, using the almost symmetry result in one direction of the previous section (Lemma \ref{s5lemma1}). This together with a simple observation on set reflections leads to an estimate on $\lambda_e = \mathrm{dist}(0,T^e)$.

\begin{lemma}
\label{s5lemma1}
Let $\varepsilon :=  \min \{ 1/4, 1/n \} \, | \Omega | / \overline{C}$ with $\overline{C}$ as in Lemma \ref{s4lemma2}. Assume that 
\begin{equation}\label{eq:smallness assumption}
	[u]_{\partial G}^{\frac{1}{s+2}} \leq \varepsilon
\end{equation}
 and suppose that the critical hyperplanes with respect to the coordinate directions $T^{e_j}$ coincide with $\{ x_j = 0 \}$ for every $j = 1, \dots , n$. For a fixed direction $e \in \mathbb{S}^{n-1}$ we have
\begin{equation}
\label{s5eq1}
    |\lambda_e| \leq \widehat{C} \, [u]_{\partial G}^{\frac{1}{s+2}}
\end{equation}
where $\widehat{C} = 4 \, (n+3) \, \frac{ \mathrm{diam}(\Omega)}{ | \Omega | } \, \overline{C} > 0$.
\end{lemma}

\begin{proof}
We set $\Omega^0 := \{ -x \ | \ x \in \Omega \}$. Since $\Omega^0$ can be obtained via composition of the $n$ reflections with respect to the hyperplanes $T^{e_j}$ for $j = \{1, \dots, n \}$, by applying Lemma \ref{s4lemma2} $n$ times with respect to the coordinate directions we obtain
\begin{equation}
\label{s5eq2}
 | \Omega \, \triangle \, \Omega^0 | \leq n \, \overline{C} \, [u]_{\partial G}^{\frac{1}{s+2}},   
\end{equation}
where we define the symmetric difference between two sets $A$ and $B$ as $A \, \triangle \, B := ( A \setminus B ) \cup ( B \setminus A)$. Indeed, we first notice that
\begin{equation*}
    | \Omega \, \triangle \, \Omega^0 | = 2 \, | \Omega \setminus \Omega^0 |.
\end{equation*}
Moreover, we have that
\begin{equation*}
    |\Omega \setminus \Omega^0| \leq | \Omega \setminus Q^n (Q^{n-1}( \dots (Q^1 (\Omega)) \dots ) | \leq |\Omega \setminus Q^n(\Omega)| + |Q^n(\Omega) \setminus Q^n (Q^{n-1}( \dots (Q^1 (\Omega)) \dots ) |, 
\end{equation*}
where $Q^j = Q^{e_j}$ the reflection with respect to the critical value in the coordinate direction $e_j$, for $j$ from $1$ to $n$. Now observing that
\begin{equation*}
    |Q^n(\Omega) \setminus Q^n (Q^{n-1}( \dots (Q^1 (\Omega)) \dots )| = |Q^n \big( \, \Omega \setminus (Q^{n-1}( \dots (Q^1 (\Omega)) \dots ) \, \big) |,
\end{equation*}
using the estimate in Lemma \ref{s4lemma2} and iterating the argument we obtain \eqref{s5eq2}.

\medskip
Now, assume $\lambda_e > 0$. 

We notice that $\Lambda_e \leq \mathrm{diam}(\Omega)$. In fact, if $\Lambda_e > \mathrm{diam}(\Omega)$, then $x \cdot e \ge 0$ for every $x \in \Omega$, and hence
$$
|\Omega \Delta \Omega^0 |= 2 |\Omega|.
$$
By using the last identity with \eqref{s5eq2}, we would find 
$$
2|\Omega| \le n \overline{C } \, [u]_{\partial G}^{\frac{1}{s+2}} ,
$$
which contradicts \eqref{eq:smallness assumption}.
 
Now let $\Omega' = Q^e (\Omega)$ be the reflection of $\Omega$ about the critical hyperplane $T^e$. Using Lemma \ref{s4lemma2} in the direction $e$ we get
\begin{equation}
\label{s5eq3}
    |\Omega \, \triangle \, \Omega'| \leq \overline{C} \, [u]_{\partial G}^{\frac{1}{s+2}}.
\end{equation}

Recalling that $\mathcal{E}_\lambda = \{ x \cdot e > \lambda \}$ and $\Omega_\lambda = \Omega \cap \mathcal{E}_\lambda$, from \eqref{s5eq3} we get
\begin{equation}
\label{s5eq4}
    |\Omega_{\lambda_e}| \geq \frac{|\Omega|}{2} - \overline{C} \, [u]_{\partial G}^{\frac{1}{s+2}}. 
\end{equation}

Moreover, if we set $\mathcal{E}_{\lambda}^0 := \{-x \ | \ x \in \mathcal{E}_{\lambda} \}$ we also have
\begin{equation*}
    |\Omega \cap \mathcal{E}_{\lambda_e}^0| = |\Omega^0 \cap \mathcal{E}_{\lambda_e}|  \geq |\Omega_{\lambda_e}| - |\Omega \,  \triangle \, \Omega^0| \geq \frac{|\Omega|}{2} - (n+1) \, \overline{C} \, [u]_{\partial G}^{\frac{1}{s+2}},
\end{equation*}
which together with \eqref{s5eq4} gives
\begin{equation}
\label{s5eq6}
| \, \{ x \in \Omega \ | \ - \lambda_e \leq x \cdot e \leq \lambda_e \} \,| \leq (n+2) \, \overline{C} \, [u]_{\partial G}^{\frac{1}{s+2}}.
\end{equation}

Since $\{ \lambda_e \leq x \cdot e \leq 3\lambda_e \}$ is mapped into $\{ |x \cdot e| \leq \lambda_e \}$ by the reflection with respect to $T_e$, using again \eqref{s5eq2} and \eqref{s5eq6} we get
\begin{align*}
| \, \{ x \in \Omega \, | \, \lambda_e < x \cdot e < 3 \lambda_e \} \, | &\leq | \, \{ x \in \Omega' \, | \, |x \cdot e| \leq \lambda_e \} \, | \leq \\
&\leq | \, \{ x \in \Omega \, | \, |x \cdot e| \leq \lambda_e \} \, | + | \, \Omega \, \triangle \, \Omega' \, | \leq (n+3) \, \overline{C} \, [u]_{\partial G}^{\frac{1}{s+2}}.
\end{align*}

Now let $m_k := | \, \{ x \in \Omega \ | \ (2k-1)\lambda_e \leq x \cdot e \leq (2k + 1) \lambda_e \} \, |$ with $k \geq 1$. By the moving plane procedure the set  $\Omega \cap T_\mu$ (seen as a subset in $\mathbb{R}^{n-1}$) is included  in $\Omega \cap T_{\mu'}$, for every $\lambda_e \leq \mu' \leq \mu$. Therefore, $m_k$ is a decreasing sequence and for every $k \geq 1$
\begin{equation*}
    m_k \leq m_1 \leq (n+3) \, \overline{C} \, [u]_{\partial G}^{\frac{1}{s+2}}. 
\end{equation*}

Now letting $k_0$ be the smallest natural number such that $(2k_0 + 1) \, \lambda_e \geq \Lambda_e$ we get
\begin{equation*}
    |\Omega_{\lambda_e}| = |\Omega \cap \{ \lambda_e \leq x \cdot e \leq \Lambda_e \}| \leq \sum_{k=1}^{k_0} m_k \leq \frac{1}{2} \bigg( \frac{\Lambda_e}{\lambda_e} + 1 \bigg) (n+3) \, \overline{C} \, [u]_{\partial G}^{\frac{1}{s+2}}
\end{equation*}
and therefore
\begin{equation*}
    |\Omega_{\lambda_e}| \, \lambda_e \leq (n+3) \, \mathrm{diam}(\Omega) \, \overline{C} \, [u]_{\partial G}^{\frac{1}{s+2}}.
\end{equation*}
In light of \eqref{s5eq4} and \eqref{eq:smallness assumption}, we have that $|\Omega_{\lambda_e}| \geq |\Omega|/4$, and \eqref{s5eq1} follows.
\end{proof}

We are now ready to complete the proof of the stability result in Theorem \ref{theorem2}.

\begin{proof}[Proof of Theorem \ref{theorem2}]
Up to a translation we can assume that the critical hyperplanes $T^{e_j}$ with respect to the $n$ coordinate directions intersect at the origin. We choose $\varepsilon > 0$ as in the proof of Lemma \ref{s5lemma1}.

Let
\begin{equation*}
    \rho_{min} := \min_{x \in \partial \Omega}|x|, \qquad \rho_{max} := \max_{x \in \partial \Omega} |x|
\end{equation*}
and $x, y \in \partial \Omega$ such that $|x| = \rho_{min}$ and $|y| = \rho_{max}$. Notice that, if $x=y$, then $\Omega$ is a ball, and the theorem trivially holds true. Thus, we assume $x \neq y$ and consider the unit vector
\begin{equation*}
    e = \frac{x - y}{|x - y|}
\end{equation*}
and the corresponding critical hyperplane $T^e$. The method of moving planes tells us that
\begin{equation}
\label{s5eq13}
    \mathrm{dist}(x,T_e) \geq \mathrm{dist}(y,T_e).
\end{equation}
Indeed, since $x = y - t e$ with $t = |x - y|$, the critical position can be reached at most when $y'$ coincides with $x$, which corresponds to the case in \eqref{s5eq13} where we have equality, while in every other case a strict inequality holds. Therefore we get
\begin{equation}
\label{s5eq14}
    \rho_{max} - \rho_{min} = |y| - |x| \leq 2 \, \mathrm{dist}(0,T_e) = 2|\lambda_e|.
\end{equation}
Clearly, $\rho (\Omega) \leq \rho_{max} - \rho_{min}$. This, together with \eqref{s5eq14} and Lemma \ref{s5lemma1} gives \eqref{s1eq8} with $C_* = 2 \widehat{C}$, if \eqref{eq:smallness assumption} holds true. On the other hand, if \eqref{eq:smallness assumption} does not hold, that is, if
$$
[ u ]_{\partial G}>\varepsilon ,
$$
then it is trivial to check that
$$
\rho(\Omega) \le \mathrm{diam}(\Omega) \le 
\left[ \frac{ \mathrm{diam}(\Omega)}{\varepsilon^{\frac{1}{s+2}}} \right] [ u ]_{\partial G}^{\frac{1}{s+2}} , 
$$
which is \eqref{s1eq8} with $C_* = \mathrm{diam}(\Omega) / \varepsilon^{1/(s+2)}$.

That is, \eqref{s1eq8} always holds true with
$$
C_*= \max \left\lbrace 2  \widehat{C} ,  \frac{ \mathrm{diam}(\Omega)}{\varepsilon^{\frac{1}{s+2}}}  \right\rbrace .
$$
As usual, the dependence on $| \Omega |$ appearing in $\widehat{C}$ and $\varepsilon$ can be removed by using \eqref{eq:trivial bounds volume}.
\end{proof}

\section{On the dependence of~$C_*$ in Theorem~\ref{theorem2}
on the diameter of~$\Omega$}\label{1A197}

A natural question is whether or not the quantitative stability result in Theorem~\ref{theorem2}
holds true with a constant~$C_*$ which is independent of the diameter of~$\Omega$.

We show with an explicit example that this is not possible.
The example is interesting in itself since it shows an ``approximate bubbling''
for remote balls. More specifically, we take~$L>10$, to be taken as large as we wish in what follows
and~$G:=B_{1/4}(-Le_1)\cup B_{1/4}(Le_1)$. We also take~$R:=3/4$
in~\eqref{IIGR}. In this way, we have that
\[ \Omega=B_1(-Le_1)\cup B_1(Le_1),\]
namely the domain is the union of two balls of unit radius located at mutual large distance.

We take~$u$ to be the corresponding torsion function as defined in~\eqref{s1eq4}.
Let also~$v$ be the solution of
\begin{equation}
\begin{cases}
(-\Delta)^s v = 1 \quad & \textmd{in} \ B_1(-Le_1),\\
v = 0 \quad &\textmd{in} \  \mathbb{R}^n \setminus B_1(-Le_1),
\end{cases}
\end{equation}
which we know to be radial.

We define~$ w:=u-v$ and we point out that
\begin{equation*}
\begin{cases}
(-\Delta)^s w = 0 \quad & \textmd{in} \ B_1(-Le_1),\\
w = u \quad &\textmd{in} \ B_1(Le_1),\\
w = 0 \quad &\textmd{in} \  \mathbb{R}^n \setminus \big(B_1(-Le_1)\cup B_1(Le_1)\big).
\end{cases}
\end{equation*}
{F}rom this and the fractional Schauder estimates in~\cite[Theorem 1.3]{MR3988080},
used here with~$k:=\ell:=0$, $f:=0$ and
$$\gamma:=\begin{cases}\displaystyle
\frac{11}{10} \quad & \textmd{if }\;\displaystyle s\not\in\left\{\frac9{20},\frac{19}{20}\right\},\\
\\ \displaystyle
\frac{13}{10} \quad & \textmd{if }\; \displaystyle s\in\left\{\frac9{20},\frac{19}{20}\right\},
\end{cases}$$
we conclude that
\begin{equation}\label{Pmreinwdoladf}
\begin{split}
\|w\|_{C^1(B_{1/2}(-Le_1))}&\le
C\int_{\mathbb{R}^n\setminus B_{1/2}(-Le_1)}\frac{|w(y)|}{|y|^{n+2s}}\,dy\\&\le
C\left[\|w\|_{L^\infty(B_1(-Le_1)\setminus B_{1/2}(-Le_1))}+
\int_{B_{1}(Le_1)}\frac{|u(y)|}{|y|^{n+2s}}\,dy\right]\\&
\le C\left[\|w\|_{L^\infty(B_1(-Le_1)\setminus B_{1/2}(-Le_1))}+
\frac{\|u\|_{L^\infty(\mathbb{R}^n)}}{L^{n+2s}}\right],
\end{split}
\end{equation}
with~$C > 0$ depending only on~$n$ and~$s$ (which we feel free to rename from line to line).

Also, using the fractional Poisson Kernel~$P$ of the ball~$B_1$
(see e.g.~\cite[Theorem~2.10]{bucur2015some}), we have that, for all~$x\in B_1$,
\begin{eqnarray*}&&
\left|w(x-Le_1)\right|=\left|
\int_{\mathbb{R}^n\setminus B_1} P(x,y) w(y-Le_1) dy
\right|\le C (1-|x|^2)^s
\int_{\mathbb{R}^n\setminus B_1} \frac{| w(y-Le_1)|}{(|y|^2-1)^s|x-y|^n} dy\\
&&\qquad=C (1-|x|^2)^s
\int_{ B_1(2Le_1)} \frac{| u(y-Le_1)|}{(|y|^2-1)^s|x-y|^n} dy
\le \frac{C\|u\|_{L^\infty(\mathbb{R}^n)}}{L^{n+2s}}.
\end{eqnarray*}
As a result,
$$ \|w\|_{L^\infty(B_1(-Le_1))}\le\frac{C\|u\|_{L^\infty(\mathbb{R}^n)}}{L^{n+2s}}.$$
{F}rom this and~\eqref{Pmreinwdoladf} we arrive at
\begin{equation}\label{Pmreinwdoladf1}
\|w\|_{C^1(B_{1/2}(-Le_1))}\le
\frac{C\|u\|_{L^\infty(\mathbb{R}^n)}}{L^{n+2s}}.
\end{equation}
Now we take~$\varphi\in C^\infty(\mathbb{R}^n,\,[0,1])$
such that~$\varphi=1$ in~$B_2(-Le_1)\cup B_2(Le_1)$
and~$\varphi=0$ outside~$B_3(-Le_1)\cup B_3(Le_1)$.
Thus, if~$x\in B_1(-Le_1)\cup B_1(Le_1)$,
\begin{eqnarray*}&&
\int_{\mathbb{R}^n} \frac{\varphi(x) - \varphi(z)}{|x-z|^{n+2s}} dz=
\int_{\mathbb{R}^n} \frac{1 - \varphi(z)}{|x-z|^{n+2s}} dz\ge
\int_{B_1((5-L)e_1)\cup B_1((L-5)e_1)} \frac{1 - \varphi(z)}{|x-z|^{n+2s}} dz\\&&\qquad
=\int_{B_1((5-L)e_1)\cup B_1((L-5)e_1)} \frac{1}{|x-z|^{n+2s}} dz\ge c,
\end{eqnarray*}
for some~$c>0$ depending only on~$n$ and~$s$.

Accordingly, we can take~$\psi:=C\varphi$ with $C$ large enough such
that~$(-\Delta)^s\psi\ge1$. Thus, by the maximum principle,
we deduce that~$u\le\psi$ and accordingly~$\|u\|_{L^\infty(\mathbb{R}^n)}\le C$.

Plugging this information into~\eqref{Pmreinwdoladf1} we conclude that
\begin{equation*}
\|w\|_{C^1(B_{1/2}(-Le_1))}\le
\frac{C}{L^{n+2s}}.
\end{equation*}
Since~$w$ is antisymmetric, this gives that
\begin{equation*}
\|w\|_{C^1(B_{1/2}(-Le_1)\cup B_{1/2}(Le_1))}\le
\frac{C}{L^{n+2s}}.
\end{equation*}
Consequently, for all~$x\ne y\in\partial B_{1/4}(-Le_1)$
(as well as for all~$x\ne y\in\partial B_{1/4}(Le_1)$),
\begin{eqnarray*}
&& \frac{|w(x) - w(y)|}{|x-y|}\le\frac{C}{L^{n+2s}}.
\end{eqnarray*}
Also, for all~$x\in\partial B_{1/4}(-Le_1)$ and~$ y\in\partial B_{1/4}(Le_1)$, we have that~$|x-y|\ge1$,
therefore
\begin{eqnarray*}
&& \frac{|w(x) - w(y)|}{|x-y|}\le |w(x)|+|w(y)|\le2\|w\|_{L^\infty(B_{1/2}(-Le_1)\cup B_{1/2}(Le_1))}
\le\frac{C}{L^{n+2s}}.
\end{eqnarray*}

As a result,
\begin{eqnarray*}&& [u]_{\partial G} \le [v]_{\partial B_{1/4}(-Le_1)\cup
\partial B_{1/4}(Le_1)} +[w]_{\partial B_{1/4}(-Le_1)\cup\partial B_{1/4}(Le_1)} \\&&\qquad
= 0+\sup_{x,y \in \partial B_{1/4}(-Le_1)\cup\partial B_{1/4}(Le_1), \, x \neq y}
\frac{|w(x) - w(y)|}{|x-y|}\le\frac{C}{L^{n+2s}}.
\end{eqnarray*}
Hence, if~\eqref{s1eq8} holded true with~$C_*$ independent of the diameter of~$\Omega$,
we would have that
$$ \rho (B_1(-Le_1)\cup B_1(Le_1)) \leq 
\frac{C}{L^\frac{n+2s}{s+2}}.$$
For this reason, there would exist~$p \in B_1(-Le_1)\cup B_1(Le_1)$
and~$t$, $s>0$ such that
$$ B_s(p)  \subset B_1(-Le_1)\cup B_1(Le_1) \subset B_t(p) $$ and
$$ |t-s|\le\frac{C}{L^\frac{n+2s}{s+2}}.$$
But necessarily~$s\le1$ and~$t\ge L$, from which a contradiction plainly follows when~$L$
is sufficiently large.

\section{Generalizations of Theorems \ref{theorem1} and \ref{theorem2}} \label{sect8}
In this section we briefly describe how Theorems \ref{theorem1} and \ref{theorem2} can be slightly generalized in the case $G$ has multiple connected components.

Let assume that $\Omega=G+B_R$, with $G$ an open bounded set with
\begin{equation} \label{G_II}
G=G_1 \cup \ldots \cup G_m \,,
\end{equation}
where $G_i$, $i=1,\ldots,m$, are the connected components of $G$ and they are such that
$$
(G_i + B_R) \cap (G_j + B_R) = \emptyset \quad \text{ for } i \neq j \,. 
$$
In this setting, the overdetermined condition \eqref{s1eq5} can be replaced by 
\begin{equation}\label{overdetII}
u=c_i \quad  \textmd{ on } \partial G_i
\end{equation} 
for some constants $c_i$, $i=1,\ldots,m$. We have the following generalization of Theorem \ref{theorem1}.
\begin{theorem}
\label{theorem1bis}
Let $G$ be as in \eqref{G_II} with $\partial G$ of class $C^1$ and set $\Omega := G + B_R$. There exists a solution $u \in C^s(\overline{\Omega})$ of \eqref{s1eq4} satisfying \eqref{overdetII}  
 if and only if $G$ (and therefore $\Omega$) is a ball.
\end{theorem}

\begin{proof}
The proof is completely analogous to the one of Theorem \ref{theorem1}. This is due to the fact that, when we apply the method of moving planes, by construction we have that the tangency point $P$ of Case 1 and its reflected $P'$ belong to the same connected component of $G$. It is clear that in Case 2 the same holds.
\end{proof}

We now discuss how to modify our argument for generalizing Theorem \ref{theorem2} in this setting. The main point is to change the definition of deficit. Indeed, in Theorem \ref{theorem2} we used the deficit 
\begin{equation*}
    [u]_{\partial G} := \sup_{x,y \in \partial G, \, x \neq y} \frac{|u(x) - u(y)|}{|x-y|} \,.
\end{equation*}
It is clear that $ [u]_{\partial G} \neq 0$ if $c_i \neq c_j$ for some $i$ and $j$ in \eqref{overdetII} and then $ [u]_{\partial G} $ cannot be used as a deficit in this setting. For this reason, we consider the deficit
\begin{equation} \label{def*}
    [u]_* := \sup_{i=1,\ldots,m}\  \sup_{x,y \in \partial G_i \atop x \neq y} \frac{|u(x) - u(y)|}{|x-y|}  \,.
\end{equation}
By using this deficit we can argue as done for Theorem \ref{theorem2} and obtain the following result. 

\begin{theorem}
\label{theorem2bis}
Let $G$ be as in \eqref{G_II} with $\partial G$ of class $C^1$ and let $\Omega := G + B_R$. Assume that $\partial \Omega$ is of class $C^2$. Let $u \in C^2(\Omega) \cap C(\mathbb{R}^n)$ be a solution of \eqref{s1eq4}. Then, we have that
\begin{equation}
\label{s1eq8bis}
\rho (\Omega) \leq C_* \, [u]_{*}^{\frac{1}{s +2}} ,
\end{equation}
where $[u]_{*}$ is given by \eqref{def*} and $C_* > 0$ is an explicit constant only depending on $n$, $s$, $R$, and the diameter $\mathrm{diam}(\Omega)$ of $\Omega$.
\end{theorem}

\begin{proof}
By using the remark noticed in the proof of Theorem \ref{theorem1bis}, the proof of the theorem is the same as the one of Theorem \ref{theorem2} and for this reason is omitted.
\end{proof}

\section*{Acknowledgments}

It is a pleasure to thank Jack Thompson for his useful
comments on a preliminary draft of this paper.

G. Ciraolo and L. Pollastro have been partially supported by the ``Gruppo Nazionale per l'Analisi Matematica, la Probabilit\`a e le loro Applicazioni'' (GNAMPA) of the ``Istituto Nazionale di Alta Matematica'' (INdAM, Italy). 

S. Dipierro, G. Poggesi and E. Valdinoci are members of AustMS. S.~Dipierro is supported by the Australian Research Council DECRA DE180100957 ``PDEs, free boundaries and applications''. G.~Poggesi is member of INdAM/GNAMPA. G.~Poggesi and E.~Valdinoci are supported by the Australian Laureate Fellowship FL190100081 “Minimal surfaces, free boundaries and partial differential equations”.

\begin{appendix}
\begin{center}
\section*{Appendices}
\end{center}

\section{Geometric remarks}\label{APP:A}

The following technical lemma has been used in the proof of Lemma~\ref{s4lemma2}.

\begin{lemma}\label{lem: a technical simple lemma}
The following relations hold true.
\begin{enumerate}[label=(\roman*)]
\item For any two open sets $A$ and $D$ in $\mathbb{R}^n$, we have that
	\begin{equation*}
		A + D = \overline{A} + D ,
	\end{equation*} 
where $\overline{A}$ is the closure of $A$.
\item In the notation introduced in~\eqref{NOTAZ}, for any point $x \in \overline{U_0} := \overline{ Q( G \cap H^+ )}$, we have that
\begin{equation*}
	B_R (x) \subset U_R \cup \left[ \Omega \cap ( H^+ \cup T )\right] .
\end{equation*}
\end{enumerate}	
\end{lemma}
\begin{proof}
(i) The inclusion $\subset$ is obvious. Let us prove $\supset$. For any $x \in \overline{A} + D $, we have that $x=a+d$, with $a\in \overline{A}$ and $d \in D$. Since $D$ is open, there exists $r_d>0$ such that $B_{r_d}(d)\subset D$. Since $a\in \overline{A}$, we can find $\underline{a} \in A$ such that $|\underline{a}-a| < r_d$. Now we notice that
$$
x=a+d= \underline{a} +(a - \underline{a} + d).
$$
Since the term in brackets belongs to $B_{r_d}(d) \subset D$ and $\underline{a} \in A$, we thus have proved that $x \in A + D  $.

(ii) For any $x \in \overline{U_0}$, we have that
$$
B_R(x) \subset \overline{U_0} + B_R(x)
$$	
by definition of $+$. An application of item (i) with $A:= U_0 $ and $D:= B_R(x)$ then gives that
$$
B_R(x) \subset U_0 + B_R(x) .
$$
The conclusion follows by noting that $U_0 \subset U_0 \cup \left[ G \cap (H^+ \cup T)\right]$.
\end{proof}

\section{Motivation for the overdetermined problem in~\eqref{s1eq4}
and~\eqref{s1eq5}:
the fair shape for an urban settlement}\label{URBE}

A classical topic in social sciences consists in the definition and understanding of the
complex transition zones (usually called ``fringes'') on the periphery of urban areas, see e.g.~\cite{101093sf472202}.
The rural-urban fringe problem aims therefore at detecting the transition
in land use and demographic characteristics lying between the continuously built-up areas of a central city
and the rural hinterland: this problem is of high social impact, also given the possible incomplete
penetration of urban utility services in fringes.

Though the analysis of fringes is still under an intense debate and several aspects, especially the ones related
to high commercial and financial pressures, are still to be considered controversial, a very simple model
could be to limit our analysis to one of the features usually attributed to fringes, namely that of {\em low density of
occupied dwellings}, and relate it to some of the characteristics that are considered inadequate for the fringe well-being
such as ``incomplete range
and incomplete network of utility services such
as reticulated water, electricity, gas and sewerage mains, fire hydrants'', etc.,
as well as ``accessibility of schools''~\cite{101093sf472202}.

One can also assume that {\em distance to urbanized areas} is a major factor to be accounted for in the analysis
of the above features since ``distance operates as a major constraint in
shaping and facilitating urban growth, and
the friction of space experienced by the rural-urban fringe is but a particular example of a
principle generally accepted in human ecology
and geography: the layout of a metropolis -- the assignment of activities to areas --
tends to be determined by a principle which
may be termed the minimizing of the cost of friction''~\cite{HAIG, 101093sf472202}.

In this spirit, one can consider a model in which the environment is described by a domain~$\Omega$ and
the density of population (or better to say the density
of occupied dwellings) is modeled by a function~$u$.
We assume that the population follows a nonlocal dispersal strategy modeled by the fractional Laplacian
(see e.g.~\cite{GIAC}) and that the environment is hostile (no dwelling possible outside the domain~$\Omega$,
with population ``killed'' if exiting the domain, corresponding to~$u=0$ outside~$\Omega$).

In this setting an equilibrium configuration for the population, subject to a growth modeled by a function~$f(x,u)$,
is described by the problem
\begin{equation}
\label{s1eq4:A}
\begin{cases}
(-\Delta)^s u(x) = f(x,u(x)) \quad & \textmd{for all } \ x\in \Omega,\\
u(x) = 0 \quad &\textmd{for all } \ x\in \mathbb{R}^n \setminus \Omega.
\end{cases}
\end{equation}
The case in which the birth and death rates of the population are negligible and the population is subject to
a constant immigration factor reduces~$f$ to a constant and therefore, up to a normalization, the problem in~\eqref{s1eq4:A}
boils down to that in~\eqref{s1eq4}.

One could also assume that there is a small quantity, say~$c>0$, that describes the density threshold
for an efficient network of utility services to develop: in this simplified model, the fringe is therefore
described by the area in which the values of~$u$ belong to the interval~$[0,c]$.

Clearly, the areas of major social hardship in this model would correspond to the points~$x$ of~$\Omega$
in the vicinity of the boundary and with~$u(x)\in[0,c]$. Assuming distance to facilities
to be the leading factor towards well-being in this simplified model, the ``fairest'' configurations
for the inhabitant of the fringe could be that in which the most remote areas are all at the same distance, say~$R$,
to the developed zone: one could therefore (at least for small~$c$ and correspondingly small~$R$)
adopt the setting in~\eqref{IIGR}.

In this framework, the above fairest condition would translate into the requirement that the density threshold~$\{u=c\}$
would coincide with~$\partial G$, leading naturally to the overdetermined condition in~\eqref{s1eq5}.

In this spirit (and with a good degree of approximation) the overdetermined problem in~\eqref{s1eq4}
and~\eqref{s1eq5} would correspond to that of a population in a hostile environment,
with negligible birth and death rate and a constant immigration factor, that adopts a nonlocal dispersal strategy
modeled by~$(-\Delta)^s$, which aims at optimizing the rural-urban fringe in terms of equal maximal density to the boundary
(the results presented here would give that the optimizer is given by a round city).
\end{appendix}

\bibliographystyle{alpha}
\bibliography{References}

\begin{thebibliography}{DNPV12b}

\bibitem[Buc16]{bucur2015some}
Claudia Bucur.
\newblock Some observations on the {G}reen function for the ball in the
  fractional {L}aplace framework.
\newblock {\em Commun. Pure Appl. Anal.}, 15(2):657--699, 2016.

\bibitem[CFMN18]{ciraolo2018rigidity}
Giulio Ciraolo, Alessio Figalli, Francesco Maggi, and Matteo Novaga.
\newblock Rigidity and sharp stability estimates for hypersurfaces with
  constant and almost-constant nonlocal mean curvature.
\newblock {\em J. Reine Angew. Math.}, 741:275--294, 2018.

\bibitem[CMS15]{CMS}
Giulio Ciraolo, Rolando Magnanini, and Shigeru Sakaguchi.
\newblock Symmetry of minimizers with a level surface parallel to the boundary.
\newblock {\em J. Eur. Math. Soc. (JEMS)}, 17(11):2789--2804, 2015.

\bibitem[CMS16]{ciraolo2016solutions}
Giulio Ciraolo, Rolando Magnanini, and Shigeru Sakaguchi.
\newblock Solutions of elliptic equations with a level surface parallel to the
  boundary: stability of the radial configuration.
\newblock {\em J. Anal. Math.}, 128:337--353, 2016.

\bibitem[CPY22]{CPY}
Lorenzo Cavallina, Giorgio Poggesi, and Toshiaki Yachimura.
\newblock Quantitative stability estimates for a two-phase serrin-type
  overdetermined problem.
\newblock {\em Nonlinear Analysis}, 222:112919, 2022.

\bibitem[DGV21]{GIAC}
Serena Dipierro, Giovanni Giacomin, and Enrico Valdinoci.
\newblock Efficiency functionals for the l\'evy flight foraging hypothesis.
\newblock {\em Preprint mp\_arc:21-25}, 2021.

\bibitem[DNPV12a]{di2012hitchhiker}
Eleonora Di~Nezza, Giampiero Palatucci, and Enrico Valdinoci.
\newblock Hitchhiker's guide to the fractional {S}obolev spaces.
\newblock {\em Bull. Sci. Math.}, 136(5):521--573, 2012.

\bibitem[DNPV12b]{MR2944369}
Eleonora Di~Nezza, Giampiero Palatucci, and Enrico Valdinoci.
\newblock Hitchhiker's guide to the fractional {S}obolev spaces.
\newblock {\em Bull. Sci. Math.}, 136(5):521--573, 2012.

\bibitem[DPTV22]{DipPogTomVald}
Serena Dipierro, Giorgio Poggesi, Jack Thompson, and Enrico Valdinoci.
\newblock The role of antisymmetric functions in nonlocal equations.
\newblock {\em Preprint arXiv:2203.11468}, 2022.

\bibitem[DSV19]{MR3988080}
Serena Dipierro, Ovidiu Savin, and Enrico Valdinoci.
\newblock Definition of fractional {L}aplacian for functions with polynomial
  growth.
\newblock {\em Rev. Mat. Iberoam.}, 35(4):1079--1122, 2019.

\bibitem[Dyd12]{dyda2012fractional}
Bart\l~omiej Dyda.
\newblock Fractional calculus for power functions and eigenvalues of the
  fractional {L}aplacian.
\newblock {\em Fract. Calc. Appl. Anal.}, 15(4):536--555, 2012.

\bibitem[DZ94]{DelfourZolesio1994}
Michel~C Delfour and Jean-Paul Zol{\'e}sio.
\newblock Shape analysis via oriented distance functions.
\newblock {\em Journal of functional analysis}, 123(1):129--201, 1994.

\bibitem[FJ15]{fall2015overdetermined}
Mouhamed~Moustapha Fall and Sven Jarohs.
\newblock Overdetermined problems with fractional {L}aplacian.
\newblock {\em ESAIM Control Optim. Calc. Var.}, 21(4):924--938, 2015.

\bibitem[Gri11]{MR3396210}
Pierre Grisvard.
\newblock {\em Elliptic problems in nonsmooth domains}, volume~69 of {\em
  Classics in Applied Mathematics}.
\newblock Society for Industrial and Applied Mathematics (SIAM), Philadelphia,
  PA, 2011.
\newblock Reprint of the 1985 original [ MR0775683], With a foreword by Susanne
  C. Brenner.

\bibitem[GT77]{gilbarg2015elliptic}
David Gilbarg and Neil~S. Trudinger.
\newblock {\em Elliptic partial differential equations of second order}.
\newblock Grundlehren der Mathematischen Wissenschaften, Vol. 224.
  Springer-Verlag, Berlin-New York, 1977.

\bibitem[Hai26]{HAIG}
R.~M. Haig.
\newblock Toward an understanding of the metropolis: Some speculations
  regarding the economic basis of urban concentration.
\newblock {\em Quarterly Journal of Economics}, 40:179--208, 1926.

\bibitem[MP19]{MP}
Rolando Magnanini and Giorgio Poggesi.
\newblock On the stability for {A}lexandrov's soap bubble theorem.
\newblock {\em J. Anal. Math.}, 139(1):179--205, 2019.

\bibitem[MP20]{MP2}
Rolando Magnanini and Giorgio Poggesi.
\newblock Serrin's problem and {A}lexandrov's soap bubble theorem: enhanced
  stability via integral identities.
\newblock {\em Indiana Univ. Math. J.}, 69(4):1181--1205, 2020.

\bibitem[MP23]{MP2021}
Rolando Magnanini and Giorgio Poggesi.
\newblock Interpolating estimates with applications to some quantitative
  symmetry results.
\newblock {\em Mathematics in Engineering}, 5(1):1--21, 2023.

\bibitem[MS10]{MagnaniniSakaguchi}
Rolando Magnanini and Shigeru Sakaguchi.
\newblock Nonlinear diffusion with a bounded stationary level surface.
\newblock {\em Ann. Inst. H. Poincar\'{e} Anal. Non Lin\'{e}aire},
  27(3):937--952, 2010.

\bibitem[Pry68]{101093sf472202}
Robin~J. Pryor.
\newblock Defining the rural-urban fringe.
\newblock {\em Social Forces}, 47(2):202--215, 1968.

\bibitem[Ser71]{serrin1971symmetry}
James Serrin.
\newblock A symmetry problem in potential theory.
\newblock {\em Arch. Rational Mech. Anal.}, 43:304--318, 1971.

\bibitem[Sha12]{MR2916825}
Henrik Shahgholian.
\newblock Diversifications of {S}errin's and related symmetry problems.
\newblock {\em Complex Var. Elliptic Equ.}, 57(6):653--665, 2012.

\end{thebibliography}

\end{document}